\documentclass[preprint,12pt]{elsarticle}
\usepackage{amsmath,amsthm,amssymb}
\usepackage{epstopdf}
\usepackage{makecell}
\usepackage{booktabs}
\usepackage{multirow}
\usepackage{diagbox}
\usepackage{extarrows}
\biboptions{numbers,sort&compress}
\usepackage{ulem}
\usepackage[colorlinks,linkcolor=blue,anchorcolor=blue,citecolor=blue]{hyperref}
\usepackage{cleveref}
\usepackage{bm,bbm}

\newtheorem{lemma}{Lemma}
\newtheorem{theorem}{Theorem}
\newtheorem{proposition}{Proposition}

\newtheorem{assumption}{Assumption}

\newtheorem{remark}{\indent {Remark}}
\newtheorem{example}{Example}
\numberwithin{theorem}{section} 
\numberwithin{equation}{section}
\numberwithin{lemma}{section}
\numberwithin{example}{section}
\numberwithin{definition}{section}

\def\dfrac{\displaystyle\frac}

\allowdisplaybreaks[4]

\journal{Journal}

\begin{document}
\begin{frontmatter}

\title{Optimal preconditioners for nonsymmetric multilevel Toeplitz systems with application to solving non-local evolutionary partial differential equations}

\author[add1]{Yuan-Yuan Huang}\ead{doubleyuans@hkbu.edu.hk}
\author[add1]{Sean Y. Hon\corref{cor1}}\ead{seanyshon@hkbu.edu.hk}
\author[add2]{Lot-Kei Chou}\ead{rickyclk16@gmail.com}
\author[add2]{Siu-Long Lei}\ead{sllei@um.edu.mo}
\cortext[cor1]{Corresponding author.}

\address[add1]{Department of Mathematics, Hong Kong Baptist University, Kowloon Tong, Hong Kong SAR.}
\address[add2]{Department of Mathematics, University of Macau, Macau SAR.}




\begin{abstract}
Preconditioning for multilevel Toeplitz systems has long been a focal point of research in numerical linear algebra. In this work, we develop a novel preconditioning method for a class of nonsymmetric multilevel Toeplitz systems, which includes the all-at-once systems that arise from evolutionary partial differential equations. These systems have recently garnered considerable attention in the literature. To further illustrate our proposed preconditioning strategy, we specifically consider the application of solving a wide range of non-local, time-dependent partial differential equations in a parallel-in-time manner. For these equations, we propose a symmetric positive definite multilevel Tau preconditioner that is not only efficient to implement but can also be adapted as an optimal preconditioner. In this context, the proposed preconditioner is optimal in the sense that it enables mesh-independent convergence when using the preconditioned generalized minimal residual method. Numerical examples are provided to critically analyze the results and underscore the effectiveness of our preconditioning strategy.
\end{abstract}
\end{frontmatter}

\section{Introduction}

Over the past few decades, preconditioning Toeplitz systems has been a major focus of research. Recently, interest in these systems has been reinvigorated, particularly for nonsymmetric Toeplitz systems, due to the symmetrization approach introduced in \cite{doi:10.1137/140974213}. In this line of research, the main strategy involves symmetrizing the original nonsymmetric Toeplitz systems into a symmetric Hankel systems using a permutation matrix, without considering its normal equations. This symmetrization enables the use of the minimal residual (MINRES) method, which facilitates convergence analysis by making the eigenvalues the primary factor in analyzing MINRES convergence. This pioneering work has led to the development of a series of effective preconditioners; see, for example, \cite{Pestana2019,Hon_SC_Wathen,hondongSC2023,FungHon_2024,LiLinHonWu2024}. Notably, the author in \cite{Pestana2019} extended this symmetrization approach to the multilevel case, and proposed an ideal multilevel Toeplitz preconditioner based on the symmetric part of the original nonsymmetric matrix. Although this approach may enhance convergence analysis, in this work, we show that neither symmetrization nor normalization is, in fact, necessary.

As in the approach developed in \cite{Pestana2019}, we focus on a class of nonsymmetric multilevel Toeplitz matrices generated by a multivariate complex-valued function $f$ with an essentially positive real part, Re($f$). For these matrices, we develop a symmetric positive definite (SPD) preconditioning strategy without resorting to symmetrization or normalization. We show that this proposed preconditioner can achieve optimal convergence using the generalized minimal residual (GMRES) method. Specifically, we consider an $n \times n$ nonsymmetric multilevel Toeplitz matrix generated by $f$, denoted as $T_n[f]$. For $T_n[f]$, we propose using its symmetric part, $\mathcal{H}(T_n[f]):=(T_n[f]+T_n[f]^\top)/2$, as a preconditioner. We then show that employing $\mathcal{H}(T_n[f])$ as a preconditioner enables the GMRES method to attain optimal convergence.

Although this ideal preconditioner $\mathcal{H}( T_n[f])$, being itself another multilevel Toeplitz matrix, is generally computationally expensive to implement directly, it can serve as a foundation for developing more efficient preconditioners. To illustrate this point and the applicability of our preconditioning framework, we consider solving non-local evolutionary partial differential equations (PDEs) as a practical application. The development of efficient parallel-in-time (PinT) iterative solvers for these PDEs has gained increasing popularity. In this study, we develop a novel PinT preconditioner, based on our proposed framework and the utilization of Tau matrices, specifically designed to address a broad spectrum of non-local evolutionary PDEs.

The paper is organized as follows. In Section \ref{sec:prelim}, preliminaries on multilevel Toeplitz and Tau matrices are presented. In Section \ref{sec:multilevel Toeplitz system}, an ideal preconditioner for multilevel Toeplitz systems and convergence analysis are presented. In Section \ref{sec:PDE_application}, a class of non-local evolutionary PDEs and how our proposed preconditioning approach can be applied for solving them effectively are discussed. The main results are given in Section \ref{sec:main}. In Section \ref{sec:numeric}, numerical results are reported for showing the effectiveness and robustness of our proposed preconditioned iterative solver.

\section{Preliminary}\label{sec:prelim}

In this section, we provide some useful background knowledge regarding multilevel Toeplitz and Tau matrices.

\subsection{Preliminaries on multilevel Toeplitz matrices}

Now consider the Banach space $L^1([-\pi,\pi]^p)$ of all complex-valued Lebesgue integrable functions over $[-\pi,\pi]^p$, equipped with the norm
\[
\|f\|_{L^1} = \frac{1}{(2\pi)^p}\int_{[-\pi,\pi]^p} |f({\boldsymbol{\theta}})|\,{\rm d} {\theta}  < \infty,
\]
where ${\rm d} {\theta}={\rm d} {\theta_1}\cdots{\rm d} {\theta_k}$ denotes the volume element with respect to the $k$-dimensional Lebesgue measure.

Let $f:$~$[-\pi,\pi]^p\to \mathbb{C}$ be a function belonging to $L^1([-\pi,\pi]^p)$ and periodically extended to $\mathbb{R}^p$. The multilevel Toeplitz matrix $T_{{n}}[f] $ of size $n\times n$ with $n= n_1n_2 \cdots n_p$ is defined as
\begin{equation*}
T_{{n}}[f] =\sum_{|j_1|<n_1}\ldots \sum_{|j_k|<n_p} J_{n_1}^{j_1} \otimes \cdots\otimes J_{n_p}^{j_p} a_{(\mathbf{j})}, \qquad \mathbf{j}=(j_1,j_2,\dots,j_p)\in \mathbb{Z}^p,
\end{equation*}
where
\[ a_{(\mathbf{j})}=a_{(j_1,\ldots, j_p)}=\frac{1}{(2\pi)^p}\int_{[-\pi,\pi]^p}f({\boldsymbol{\theta}}){\rm e}^{-\mathbf{i}\left\langle { \bf j},{\boldsymbol{\theta}}\right\rangle}\, {\rm d}{\theta},
\qquad \left\langle { \bf j},{\boldsymbol{\theta}}\right\rangle=\sum_{t=1}^p j_t\theta_t, \quad \mathbf{i}^2=-1, \]
are the Fourier coefficients of $f$ and $J^{j}_{m}$ is the $m \times m$ matrix whose $(l,h)$-th entry equals 1 if $(l-h)=j$ and $0$ otherwise. 

The function $f$ is called the generating function of $T_{{n}}[f]$. It is easy to prove that (see e.g., \cite{ng2004iterative,MR2376196,Chan:1996:CGM:240441.240445,GaroniCapizzano_two}) if $f$ is real-valued, then $T_{{n}}[f]$ is Hermitian; if $f$ is real-valued and nonnegative, but not identically zero almost everywhere, then $T_{{n}}[f]$ is Hermitian positive definite; if $f$ is real-valued and even, $T_{{n}}[f]$ is (real) symmetric.



Throughout this work, we assume that $f\in L^1([-\pi,\pi]^p)$ and is periodically extended to $\mathbb{R}^p$.

\begin{lemma}(see, e.g., \cite[Theorem 2.4]{Serra_94})\label{lemma:Toeplitz_localization}
    Let $f,g\in L^1([-\pi,\pi]^p)$ be real-valued functions with $g$ essentially positive. Let
    \[
    r:= \operatorname{ess\,inf}_{\boldsymbol{\theta}  \in [-\pi, \pi]^p} \frac{f(\boldsymbol{\theta})}{g(\boldsymbol{\theta})}, \quad R:= \operatorname{ess\,sup}_{\boldsymbol{\theta}  \in [-\pi, \pi]^p} \frac{f(\boldsymbol{\theta})}{g(\boldsymbol{\theta})}.
    \]
    Then, the eigenvalues of $T_n^{-1}[g] T_n[f]$ lie in $(r, R)$ if $r < R$. If $r=R$, then $T_n^{-1}[g] T_n[f]=I_n$, where $I_n$ is the identity matrix of dimension $n=n_1 n_2 \cdots n_p$.
\end{lemma}

\subsection{Tau matrices as preconditioners}\label{sec:Tau}

For a symmetric Toeplitz matrix $T_m\in\mathbb{R}^{m\times m}$ with $(t_1,t_2,\ldots,t_m)^{\top}\in\mathbb{R}^{m}$, define its $\tau$-matrix \cite{Bini1990} approximation as
\begin{equation}\label{tauopdef}
    \tau(T_m):=T_m-H_m,
\end{equation}
where $H_m$ is the Hankel matrix with $(t_3,t_4,...,t_m,0,0)^{\top}$ as its first column and $(0,0,t_m,...,t_4,t_3)^{\top}$ as its last column. A crucial property of the Tau matrix defined in \eqref{tauopdef}  is that it is diagonalizable by sine transform matrix, i.e.,
\begin{equation}\label{taumatdiag}
    \tau(T_m)=S_m Q_m S_m,
\end{equation}
where $Q_m=[{\rm diag}(q_i)]_{i=1}^{m}$ is a diagonal matrix with
\begin{equation}\label{sigmicomp}
    q_{i}=t_1+2\sum\limits_{j=2}^{m}t_j\cos\left(\frac{\pi i(j-1)}{m+1}\right),\quad  i\in 1\wedge m.
\end{equation} 
\begin{equation*}
    S_m:= \left[\sqrt{\frac{2}{m+1}}\sin\left(\frac{\pi jk}{m+1}\right)\right]_{j,k=1}^{m}
\end{equation*}
is a sine transform matrix. It is easy to verify that $S_m$ is a symmetric orthogonal matrix, i.e., $S_m=S_m^{\top}=S_m^{-1}$. The product of matrix $S_m$ and a given vector of length $m$ can be fast computed in $\mathcal{O}(m\log m)$ operations using discrete sine transforms (DSTs) \cite{BC83}. Let ${\bf e}_{m,i}\in\mathbb{R}^{m}$ denote the $i$-th column of the $m\times m$ identity matrix.
We also note that the $m$ numbers $\{q_i\}_{i=1}^{m}$ defined in \eqref{sigmicomp} can be computed by \cite{huang2023tau}
\begin{equation*}
    (q_1, q_2, \cdots, q_m)^{\top}={\rm diag}(S_m{\bf e}_{m,1})^{-1}[S_m\tau(T_m){\bf e}_{m,1}].
\end{equation*}
From the equation above, we know that the computation of $\{q_i\}_{i=1}^{m}$ requires only $\mathcal{O}(m\log m)$ operations.

For a real square matrix $Z$, denote the symmetric part and skew-symmetric part of $Z$ as
$$
\mathcal{H}(Z):=\frac{Z+Z^{\top}}{2} \quad \textrm{and} \quad \mathcal{S}(Z):=\frac{Z-Z^{\top}}{2}.
$$

\section{Ideal preconditioner for multilevel Toeplitz systems $T_n[f]$x=b} \label{sec:multilevel Toeplitz system}
Consider solving the following nonsymmetric multilevel Toeplitz systems
\begin{equation}\label{general_multilevel_sys}
	A_n {\bf u}={\bf f},
\end{equation}
where $A_n:=T_n[f]$ is generated by a multivariate complex-valued function $f$ with essentially positive real part Re($f$).
If we use $A_R:=\mathcal{H}(T_n[f])=(A_{n}+A_{n}^{\top} ) / 2$ as a preconditioner to accelerate the convergence of Krylov subspace methods such as GMRES method when solving (\ref{general_multilevel_sys}), the 
 corresponding preconditioned systems will be
\begin{equation}\label{one_sided_sys}
	\mathcal{H}(T_n[f])^{-1} T_n[f] {\bf u}= \mathcal{H}(T_n[f])^{-1}{\bf f}.
\end{equation}

We note that Re($f$) is essentially positive, which indicates that $A_R$ is SPD \cite{Pestana2019} and hence $A_R^{-\frac{1}{2}}$ exists \cite{lin2021all}. In order to show the effectiveness of the preconditioner $A_R$ for the one-sided preconditioned system (\ref{one_sided_sys}), we introduce the following auxiliary two-sided preconditioned system
\begin{align}\label{two_sided_sys}
& A_R^{-\frac{1}{2}} A_n \underbrace{A_R^{-\frac{1}{2}}\mathbf{\hat u}}_{=:\mathbf{u}}=A_R^{-\frac{1}{2}} \mathbf{f}.
\end{align}

Before showing our main preconditioning result, we first provide some useful lemmas in what follows.

\begin{lemma}\label{lemma:wghtsumbdlem}
	For nonnegative numbers $\xi_i$ and positive numbers $\zeta_i$ $(1\leq i\leq m)$, it holds that
\begin{equation*}
\min\limits_{1\leq i\leq m}\frac{\xi_i}{\zeta_i}\leq\bigg(\sum\limits_{i=1}^{m}\zeta_i\bigg)^{-1}\bigg(\sum\limits_{i=1}^{m}\xi_i\bigg)\leq\max\limits_{1\leq i\leq m}\frac{\xi_i}{\zeta_i}.
	\end{equation*}
\end{lemma}

\begin{lemma}\cite[Proposition 7.3]{ElmanSilvesterWathen2014}\label{lemma:gmres}
    Let $Z\mathbf{v}=\mathbf{w}$ be a real square linear system with $\mathcal{H}(Z) \succ \mathcal{O}$. Then, the residuals of the iterates generated by applying (restarted or non-restarted) GMRES with an arbitrary initial guess to solve $Z\mathbf{v}=\mathbf{w}$ satisfy
    \[
        \| \mathbf{r}_k \|_2 \leq \left( 1 - \frac{\lambda_{\min}( \mathcal{H}(Z)  )^2 }{ \lambda_{\min}( \mathcal{H}(Z)  ) \lambda_{\max}( \mathcal{H}(Z)  ) + \rho( \mathcal{S}(Z)   )^2   }   \right)^{k/2} \| \mathbf{r}_0 \|_2,
    \]
    where $\mathbf{r}_k = \mathbf{w} - Z\mathbf{v}_k$ is the residual vector at the $k$-th GMRES iteration with $\mathbf{v}_k$ ($k \geq 1$) being the corresponding iterative solution.
\end{lemma}

\begin{lemma}\label{lemma:gmres_residual}
    For a non-singular $n \times n$ real linear system $A\mathbf{y}=\mathbf{b}$, let $\mathbf{y}_j$ be the iteration solution by GMRES at the $j$-th $(j \geq 1)$ iteration step with $\mathbf{y}_0$ as initial guess. Then, the $j$-th iteration solution $\mathbf{y}_j$ minimizes the residual error over the Krylov subspace $\mathcal{K}_j\left(A, \mathbf{r}_0\right)$ with $\mathbf{r}_0=\mathbf{b}-A \mathbf{y}_0$, i.e.,
    \begin{equation*}
\mathbf{y}_j=\underset{\mathbf{v} \in \mathbf{y}_0+\mathcal{K}_j\left(A, \mathbf{r}_0\right)}{\arg \min }\|\mathbf{b}-A\mathbf{v}\|_2 .
\end{equation*}
\end{lemma}

The following lemma shows that the convergence of the GMRES solver for (\ref{one_sided_sys}) is supported by the convergence of that for (\ref{two_sided_sys}).

\begin{lemma}\label{residual_relation}
Let $\hat{\mathbf{u}}_0$ be the initial guess for (\ref{two_sided_sys}) and $\mathbf{u}_0 := A_R^{-1/2}\hat{\mathbf{u}}_0$ be the initial guess for (\ref{one_sided_sys}). Let $\mathbf{u}_j$ ($\hat{\mathbf{u}}_j$, respectively) be the $j$-th $(j\geq1)$ iteration solution derived by applying GMRES solver to (\ref{one_sided_sys}) ((\ref{two_sided_sys}), respectively) with $\mathbf{u}_0$ ($\hat{\mathbf{u}}_0$, respectively) as their initial guess. Then,
\begin{equation*}
\left\|\mathbf{r}_j\right\|_2 \leq \frac{1}{\sqrt{\lambda_{\min}(A_R)}}\left\|\hat{\mathbf{r}}_j\right\|_2
\end{equation*}
where $\mathbf{r}_j:=A_R^{-1} \mathbf{f} - A_R^{-1} A_n\mathbf{u}_j$ ($\hat{\mathbf{r}}_j := A_R^{-1/2} \mathbf{f} - A_R^{-1/2} A_n A_R^{-1/2} \hat{\mathbf{u}}_j$, respectively) denotes the residual vector at the $j$-th GMRES iteration for (\ref{one_sided_sys}) ((\ref{two_sided_sys}), respectively).
\end{lemma}
\begin{proof}
The direct application of Lemma \ref{lemma:gmres_residual} to (\ref{two_sided_sys}) leads to 
\begin{equation*}
\begin{aligned}
\hat{\mathbf{u}}_j-\hat{\mathbf{u}}_0 &\in \mathcal{K}_j\left(A_R^{-\frac{1}{2}} A_n A_R^{-\frac{1}{2}}, \hat{\mathbf{r}}_0\right)  ~~~~~~~ (\hat{\mathbf{r}}_0=A_R^{-\frac{1}{2}} \mathbf{f}- A_R^{-\frac{1}{2}} A_n A_R^{-\frac{1}{2}}\hat{\mathbf{u}}_0)\\
& =\operatorname{span}\left\{\left(A_R^{-\frac{1}{2}} A_n A_R^{-\frac{1}{2}}\right)^k\left(A_R^{-\frac{1}{2}} \mathbf{f}- A_R^{-\frac{1}{2}} A_n A_R^{-\frac{1}{2}}\hat{\mathbf{u}}_0\right)\right\}_{k=0}^{j-1} \\
& =\operatorname{span}\left\{A_R^{\frac{1}{2}}\left(A_R^{-1} A_n\right)^k A_R^{-\frac{1}{2}}\left(A_R^{-\frac{1}{2}} \mathbf{f}- A_R^{-\frac{1}{2}} A_n A_R^{-\frac{1}{2}}\hat{\mathbf{u}}_0\right)\right\}_{k=0}^{j-1} \\
& =\operatorname{span}\left\{A_R^{\frac{1}{2}}\left(A_R^{-1}A\right)^k\left(A_R^{-1} \mathbf{f}- A_R^{-1} A_n{\mathbf{u}}_0\right)\right\}_{k=0}^{j-1}.
\end{aligned}
\end{equation*}
Then, we have
\begin{eqnarray*}
A_R^{-\frac{1}{2}} \hat{\mathbf{u}}_j-\mathbf{u}_0&=&A_R^{-\frac{1}{2}}\left(\hat{\mathbf{u}}_j-\hat{\mathbf{u}}_0\right)\\
&\in& \operatorname{span}\left\{\left(A_R^{-1}A_n\right)^k\left(A_R^{-1} \mathbf{f}- A_R^{-1} A_n{\mathbf{u}}_0\right)\right\}_{k=0}^{j-1}\\
&=&\mathcal{K}_j\left(A_R^{-1}A_n, \mathbf{r}_0\right),
\end{eqnarray*}
which means
\begin{equation*}
A_R^{-\frac{1}{2}} \hat{\mathbf{u}}_j \in \mathbf{u}_0+\mathcal{K}_j\left(A_R^{-1}A_n, \mathbf{r}_0\right) .
\end{equation*}
Again, for (\ref{one_sided_sys}), Lemma \ref{lemma:gmres_residual} indicates that
\begin{equation*}
\mathbf{u}_j=\underset{\mathbf{v} \in \mathbf{u}_0+\mathcal{K}_j\left(A_R^{-1} A_n, \mathbf{r}_0\right)}{\arg \min }\left\|A_R^{-1} \mathbf{f}- A_R^{-1} A_n{\mathbf{v}}\right\|_2 .
\end{equation*}
Therefore,
\begin{equation*}
\begin{aligned}
\left\|\mathbf{r}_j\right\|_2&=\left\|A_R^{-1} \mathbf{f}- A_R^{-1} A_n \mathbf{u}_j\right\|_2\\
&\leq\left\|A_R^{-1} \mathbf{f}-A_R^{-1} A_n A_R^{-1} \hat{\mathbf{u}}_j\right\|_2 \\
& =\left\|A_R^{-\frac{1}{2}} \hat{\mathbf{r}}_j\right\|_2 \\
& =\sqrt{\hat{\mathbf{r}}_j^{\top} A_R^{-1} \hat{\mathbf{r}}_j} \\
&\leq\frac{1}{\sqrt{\lambda_{\min}(A_R)}}\left\|\hat{\mathbf{r}}_j\right\|_2.
\end{aligned}
\end{equation*}
\end{proof}

We are ready to provide the following result, which is applicable to a general multilevel Toeplitz matrix generated by a class of multivariate functions $f$.

\begin{theorem}\label{conv_gmres_1}
    Let $f \in L^1([-\pi, \pi]^p)$ and let $f=\mathrm{Re}(f)+ \mathbf{i} \mathrm{Im}(f)$, where $\mathrm{Re}(f)$ and $\mathrm{Im}(f)$ are real-valued functions with $\mathrm{Re}(f)$ essentially positive. Additionally, let $A_{n}:=T_{{n}}[f] \in \mathbb{R}^{n\times n}$ be the multilevel Toeplitz matrix generated by $f$ and let $A_R:=\mathcal{H}(  A_{n})=(A_{n}+A_{n}^{\top} ) / 2$. Then, the residuals of the iterates generated by applying (restarted or non-restarted) GMRES with an arbitrary initial guess to solve $A_R^{-\frac{1}{2}}  A_{n} A_R^{-\frac{1}{2}}\mathbf{v}=A_R^{-\frac{1}{2}}\mathbf{w}$ satisfy
    \[
       \| \mathbf{r}_k \|_2 \leq  \omega^{k} \| \mathbf{r}_0 \|_2,
    \]
    where $\mathbf{r}_k = A_R^{-\frac{1}{2}}\mathbf{w} - A_R^{-\frac{1}{2}}A_{n} A_R^{-\frac{1}{2}}\mathbf{v}_k$ is the residual vector at the $k$-th GMRES iteration with $\mathbf{v}_k$ ($k \geq 1$) being the corresponding iterative solution, and $\omega$ is a constant independent of $\mathbf{n}$ defined as follows
        $$
        \omega := \frac{\epsilon }{\sqrt{1 + \epsilon^2}} \in (0,1),
        $$
        with
    $$
    \epsilon=\underset{{\boldsymbol{\theta}  \in [-\pi, \pi]^p}}{\operatorname{ess\,sup}} \left|\frac{\mathrm{Im}(f)(\boldsymbol{\theta}) }{\mathrm{Re}(f)(\boldsymbol{\theta}) }\right|.
    $$
\end{theorem}
\begin{proof}
    First of all, since 
    \begin{eqnarray*}
        \mathcal{H} \left( \mathcal{H}(A_{n})^{-\frac{1}{2}}A_{n} \mathcal{H}(  A_{n})^{-\frac{1}{2}} \right)
        &=& A_R^{-\frac{1}{2}} \mathcal{H} (A_{n}) A_R^{-\frac{1}{2}}\\
        &=& I_n\\
        &\succ& \mathcal{O},
    \end{eqnarray*}
         Lemma \ref{lemma:gmres} can be used.

    Also, $\mathcal{S} ( A_R^{-\frac{1}{2}}A_{n}A_R^{-\frac{1}{2}} )=A_R^{-\frac{1}{2}} \mathcal{S} (A_{n}) A_R^{-\frac{1}{2}}$, 
    we know by Lemma \ref{lemma:Toeplitz_localization} that
        \begin{eqnarray*}
        \rho \left( A_R^{-\frac{1}{2}} \mathcal{S} (A_{n}) A_R^{-\frac{1}{2}} \right)&=&\rho \left( T_n[\mathrm{Re}(f)]^{-\frac{1}{2}} \left( \textbf{i} T_n[\mathrm{Im}(f)] \right)T_n[\mathrm{Re}(f)]^{-\frac{1}{2}} \right)\\
        &=&\rho \left( T_n[\mathrm{Re}(f)]^{-\frac{1}{2}} \left( T_n[\mathrm{Im}(f)] \right)T_n[\mathrm{Re}(f)]^{-\frac{1}{2}} \right)\\
        &\leq& \epsilon:=\operatorname{ess\,sup}_{\boldsymbol{\theta}  \in [-\pi, \pi]^p}\left|\frac{\mathrm{Im}(f)(\boldsymbol{\theta}) }{\mathrm{Re}(f)(\boldsymbol{\theta}) }\right|.
        \end{eqnarray*}

    Thus, by Lemma \ref{lemma:gmres}, the residuals of the iterates generated by applying (restarted or non-restarted) GMRES with an arbitrary initial guess to solve $A_R^{-\frac{1}{2}}  A_{n} A_R^{-\frac{1}{2}}\mathbf{v}=A_R^{-\frac{1}{2}}\mathbf{w}$ satisfy
    \begin{eqnarray*}
    \| \mathbf{r}_k \|_2 &\leq& \left( \sqrt{1 - \frac{1}{1 + \epsilon^2} }     \right)^{k} \| \mathbf{r}_0 \|_2\\
    &=& \left( \frac{\epsilon}{\sqrt{1 + \epsilon^2}}   \right)^{k} \| \mathbf{r}_0 \|_2.
     \end{eqnarray*}
\end{proof}

As a consequence of Theorem \ref{conv_gmres_1}, $\mathcal{H}(A_{n})$ can serve as an ideal preconditioner for $A_{n}$, despite its implementation challenges. Its optimal preconditioning efficacy is governed by the intrinsic quantity $\epsilon = \operatorname{ess\,sup}_{\boldsymbol{\theta} \in [-\pi, \pi]^p} \left|\frac{\mathrm{Im}(f)(\boldsymbol{\theta})}{\mathrm{Re}(f)(\boldsymbol{\theta})}\right|$, which depends solely on the function $f$. The smaller the value of $\epsilon$, the more effective the preconditioning effect for $A_{n}$. This result is broad in scope, suggesting that $\mathcal{H}(A_{n})$ can serve as a blueprint for the development of efficient preconditioners for various applications. To elucidate this point, the subsequent section will demonstrate how our proposed preconditioning approach can be effectively applied to solving a class of non-local evolutionary partial differential equations.

\section{Applications to non-local evolutionary partial differential equations}\label{sec:PDE_application}

Consider the following non-local evolutionary PDEs with weakly singular kernels
\begin{equation}\label{main_problem}
	\begin{cases}
		\frac{1}{\Gamma(1-\alpha)}\int_{0}^{t}\frac{\partial u(\boldsymbol{x},s)}{\partial s}(t-s)^{-\alpha}ds=\mathcal{L}u(\boldsymbol{x},t)+f(\boldsymbol{x},t),\quad \boldsymbol{x}\in \Omega\subset\mathbb{R}^d,~ t\in(0,T],\\
		u(\boldsymbol{x},t)=0,\quad \boldsymbol{x}\in\partial \Omega,~t\in(0,T],\\
		u(\boldsymbol{x},0)=\psi(\boldsymbol{x}),\quad \boldsymbol{x}\in \Omega,
	\end{cases}
\end{equation}
where $\Gamma(\cdot)$ is the Gamma function, $0< \alpha< 1$,  $f$ and $\psi$ are both given functions; the boundary of $\Omega$ is $\partial\Omega$; $\Omega=\prod_{i=1}^{d}(\check{a}_i,\hat{a}_i)$;  $\boldsymbol{x}=(x_1,x_2,...,x_d)$ is a point in $\mathbb{R}^d$; for $i=1,\dots,d$; the spatial operator $\mathcal{L}$ can be (but not limit to) the following choices
\begin{equation*}
\mathcal{L}=\left\{
\begin{aligned}
&\Delta,~~~~~\text{constant Laplacian},\\
&\sum_{i=1}^d c_i \frac{\partial^{\beta_i}}{\partial\left|x_i\right|^{\beta_i}},~~~~~\text{Riesz derivative with $\beta_i \in (1,2)$,}\\
&\sum_{i=1}^d \left(k_{i,+} \frac{\partial^{\beta_i}}{\partial_{+} x_i^{\beta_i}}+k_{i,-}\frac{\partial^{\beta_i}}{\partial_{-} x_i^{\beta_i}}\right),~~~~~\text{Riemann-Liouville derivative with $\beta_i \in (1,2)$.}
\end{aligned}
\right.
\end{equation*}

Numerical methods for solving evolutionary PDEs in the form of (\ref{main_problem}) with the aforementioned exemplary selections of spatial operator $\mathcal{L}$ have been intensively discussed; see \cite{jin2016analysis,gan2024efficient,zhu2024highly,chen2009finite,lin2024single,lin2018separable,pang2016fourth,lin2019fast,vong2016high} and the references therein.

\subsection{Temporal discretization}
For any nonnegative integer $m$, $n$ with $m\leq n$, define the set $m\wedge n:=\{m,m+1,...,n-1,n\}$. 

Denote by $\mathbb{N}^{+}$ the set of all positive integers. Let $N\in\mathbb{N}^{+}$ and the temporal stepsize $\mu=T/N$. With the L1 scheme (see, e.g.,{ \cite{sun2006fully,lin2007finite,liao2018sharp,jin2016analysis})}, the temporal discretization has the following form
\begin{align}
	&	\frac{1}{\Gamma(1-\alpha)}\int_{0}^{t}\frac{\partial u(\boldsymbol{x},s)}{\partial s}(t-s)^{-\alpha}ds\Big|_{t=n\mu}= \sum\limits_{k=1}^{n}l_{n-k}^{(\alpha)}u(\boldsymbol{x},k\mu)+l^{(n,\alpha)}\psi(\boldsymbol{x})+\mathcal{O}(\mu^{2-\alpha}),\label{convlqdra}
\end{align}
where $\boldsymbol{x}$ $ \in \Omega$, $k\in 1\wedge N$ and 
\begin{align*}
	&l_{k}^{(\alpha)}=\begin{cases}
		\kappa a_0,\quad k=0,\\
		\kappa(a_k-a_{k-1}),\quad k\in 1\wedge (N-1),
	\end{cases}\notag\\
	&l^{(n,\alpha)}=-\kappa a_{n-1},\quad n\in 1\wedge N,\notag 
\end{align*}
with $\kappa=\frac{1}{\Gamma(2-\alpha)\mu^{\alpha}}$ and $a_j=(j+1)^{1-\alpha}-j^{1-\alpha}$, $j \geq 0$.

\subsection{All-at-once system}
Suppose a uniform spatial discretization with stepsize $h_i=(\hat{a}_i-\check{a}_i)/(m_i+1)$ for $k\in 1\wedge d$ is adopted, and the spatial discretization matrix $G_J\in \mathbb{R}^{J \times J} $ with $J=\prod\limits_{i=1}^{d}m_i$ for $-\mathcal{L}$ is a multilevel ($d$-level) Toeplitz matrix associated with the function $w_{\boldsymbol{\beta}}(\boldsymbol{\theta})$ and $\mathcal{H}(G)$ is SPD. 

Denote
\begin{equation*}
	m_1^{-}=m_d^{+}=1,\quad m_i^{-}=\prod\limits_{j=1}^{i-1}m_i,~i\in 2\wedge d,\quad  m_k^{+}=\prod\limits_{j=k+1}^{d}m_j,~k\in 1\wedge (d-1).
\end{equation*}
With the spatial discretization matrix $G_J$ and the temporal discretization \eqref{convlqdra}, we obtain the following all-at-once linear system as a discretization of the problem \eqref{main_problem}
\begin{equation}\label{aaosystem}
	A{\bf u}:= \left( G_J \otimes I_N + I_J \otimes \kappa B^{(\alpha)}_N \right) {\bf u}={\bf f},
\end{equation}
where ${\bf u}=({\bf u}_1;{\bf u}_2;\dots;{\bf u}_{J}) \in \mathbb{R}^{NJ\times 1}$, ${\bf f}=({\bf f}_1;{\bf f}_2;\dots;{\bf f}_{J}) \in \mathbb{R}^{NJ\times 1}$;
$I_{k}$ denotes a $k\times k$ identity matrix; the lower triangular Toeplitz matrix  $B^{(\alpha)}_N\in\mathbb{R}^{N\times N}$ denotes the temporal discretization matrix, 
namely,
\begin{eqnarray}
\label{eqn:matrix_time_fractional}\nonumber
B^{(\alpha)}_N &:=&\left[
\begin{array}
[c]{ccccc}
l_0^{(\alpha)} & & & &\\
l_1^{(\alpha)}  & l_0^{(\alpha)} & & &\\
\vdots&\ddots &\ddots& &\\
l_{N-2}^{(\alpha)} & \ldots&\ddots& \ddots & \\
l_{N-1}^{(\alpha)} & l_{N-2}^{(\alpha)}& \ldots & l_1^{(\alpha)} & l_0^{(\alpha)}
\end{array}
\right]\\
&=&\left[
\begin{array}
[c]{ccccc}
a_0 & & & &\\
a_1-a_0 & a_0 & & &\\
\vdots&\ddots &\ddots& &\\
a_{N-2}-a_{N-3} & \ldots&\ddots& \ddots & \\
a_{N-1}-a_{N-2} & a_{N-2}-a_{N-3}& \ldots & a_1-a_0 & a_0
\end{array}
\right].
\end{eqnarray}

The generating function \cite{ng2004iterative,jin2010preconditioning} of $B^{(\alpha)}_N$ is given by
\begin{equation*}
g_{\alpha}( \phi )=a_0+\sum_{j=1}^{\infty}\left(a_j-a_{j-1}\right) e^{\mathbf{i} j \phi}.
\end{equation*}
Thus, the all-at-once matrix $A$ is associated with the following $d+1$-variate complex-valued function
\begin{equation}\label{eqn:main_gen_function}
    f_{\alpha, \boldsymbol{\beta}}(\phi,\boldsymbol{\theta}) = w_{\boldsymbol{\beta}}(\boldsymbol{\theta}) + \kappa g_{\alpha}(\phi),
\end{equation} in the sense that 
\begin{equation*}
    A = T_{J}[w_{\boldsymbol{\beta}}(\boldsymbol{\theta})]  \otimes I_N + I_J \otimes \kappa T_{N}[g_{\alpha}(\phi)].   
\end{equation*}
Evidently, for a fixed matrix size (i.e., when both $N$ and $J$ are kept fixed), the elements of $A$ are determined by the Fourier coefficients of $f_{\alpha, \boldsymbol{\beta}}(\phi,\boldsymbol{\theta})$.

Even though the matrix $G_J$ may be symmetric, the all-at-once matrix $A$ in (\ref{aaosystem}) is nonsymmetric due to the fact that $B^{(\alpha)}_N$ is a lower triangular matrix. Consequently, some commonly used Krylov subspace methods, such as the conjugate gradient method and MINRES, are not directly applicable. Therefore, to address this issue, we employ the GMRES method to solve the nonsymmetric linear system.

As the direct application of Section \ref{sec:multilevel Toeplitz system}, we follow our proposed preconditioning approach and construct the following preconditioner for $A$: 
\begin{equation}\label{Hermitian_pre}
    \hat P := \mathcal{H}(A)=  \mathcal{H}(G_J) \otimes I_N + I_J \otimes \kappa \mathcal{H}(B^{(\alpha)}_N).    
\end{equation}
We note that $\mathcal{H}(B^{(\alpha)}_N)$ is strictly diagonally dominant and hence SPD \cite{lin2018separable}, then $\hat P$ is also SPD and hence $\hat P^{-\frac{1}{2}}$ exists \cite{lin2021all}. The convergence of GMRES solver with $\hat P$ will be discussed in subsection \ref{sec:ideal_sym}.

Since $\hat P$ cannot be efficient implemented, in subsection \ref{sec:modified_P}, we will further propose a more practical preconditioner $P$ for $A$ and show that $P$ does not only facilitates efficient implementation but also leads to optimal convergence when GMRES is employed, aligning our proposed preconditioning theory (i.e., Theorem \ref{conv_gmres_1}).

\section{Main results}\label{sec:main}

The main results are divided into the following subsections.

\subsection{Convergence analysis of the ideal preconditioner $\hat P$}\label{sec:ideal_sym}

Before giving the result of the proposed preconditioner $\mathcal{H}( T_n[f] )$, the following lemma about $g_{\alpha}( \phi )$ and assumption on the function $w_{\boldsymbol{\beta}}(\boldsymbol{\theta})$ are needed to facilitate our analysis.

\begin{lemma}\label{l1_quotient}
For $\alpha\in(0,1)$ and $\forall \phi \in \mathbb{R} \backslash\{2k\pi \mid k \in \mathbb{Z}\}$, $\operatorname{Re}(g_{\alpha}( \phi ))>0$ and
\begin{equation*}
\frac{\left|\operatorname{Im}(g_{\alpha}( \phi ))\right|}{\operatorname{Re}(g_{\alpha}( \phi ))}<\tan{\left( \frac{\alpha}{2} \pi \right)}.
\end{equation*}
\end{lemma}
 \begin{proof}
For $\alpha\in \mathbbm{C}$ with $\textrm{Re}(\alpha)>0$, it is ready to show that $|a_{j}-a_{j-1}|=\mathcal{O}(j^{-1-\textrm{Re}(\alpha)})$, such that $g_{\alpha}( \phi )$ is analytic on $\{s\in\mathbbm{C}\big|\textrm{Re}(s)>0\}$ with respect to $\alpha$. Hence with the analytic continuation of the Lerch transcendent \cite{bateman1953higher} $\Phi(z,s,\nu)=\sum_{k=0}^{\infty}(k+\nu)^{-s}z^{k}$ that
\[\Phi\vert_{S_{1}}(z,s,\nu)=z^{-\nu}\Gamma(1-s)\sum\limits_{n=-\infty}^{\infty}(-\log z+2n\pi \mathbf{i})^{s-1}e^{2n\pi \mathbf{i}\nu},\]
for $z\in\mathbb{C}\setminus[1,\infty),~\lvert\arg(-\log z+2n\pi \mathbf{i})\rvert\leq\pi,~0<\nu\leq 1,~s\in S_{1}=\{s\in\mathbb{C}\big\vert\textrm{Re}~s<0\}$, we have for $\alpha\in\{s\in\mathbb{C}\big\vert 0<\textrm{Re}~s<1\},~\phi\in(0,\pi]$,
\begin{align*}
&g_{\alpha}( \phi )\\
&=(1-e^{\mathbf{i}\phi})^{2}\Phi(e^{\mathbf{i}\phi},\alpha-1,1)\\
&=-4\sin^{2}{\left(\frac{\phi}{2}\right)}\Gamma(2-\alpha)\sum\limits_{n=-\infty}^{\infty}(-\mathbf{i}\phi+2n\pi \mathbf{i})^{\alpha-2}\\
&=-4\sin^{2}{\left(\frac{\phi}{2}\right)}\Gamma(2-\alpha)\sum\limits_{n=0}^{\infty}\Big((2n\pi+\phi)^{\alpha-2}e^{-\frac{\mathbf{i}(\alpha-2)\pi}{2}}+(2(n+1)\pi-\phi)^{\alpha-2}e^{\frac{\mathbf{i}(\alpha-2)\pi}{2}}\Big)\\
&=4\sin^{2}{\left(\frac{\phi}{2}\right)}\Gamma(2-\alpha)\sum\limits_{n=0}^{\infty}\Big((2n\pi+\phi)^{\alpha-2}+(2(n+1)\pi-\phi)^{\alpha-2}\Big)\cos{\left( \frac{\alpha}{2} \pi \right)}\\
&\hspace{30pt}-\mathbf{i}4\sin^{2}{\left(\frac{\phi}{2}\right)}\Gamma(2-\alpha)\sum\limits_{n=0}^{\infty}\Big((2n\pi+\phi)^{\alpha-2}-(2(n+1)\pi-\phi)^{\alpha-2}\Big)\sin{\left( \frac{\alpha}{2} \pi \right)}.
\end{align*}
In particular for $\alpha\in(0,1)$, we have $\forall\phi\in(0,\pi]$,
\begin{eqnarray*}
&&\operatorname{Re}(g_{\alpha}( \phi ))\\
&=&4 \sin^2{\left(\frac{\phi}{2}\right)}\Gamma(2-\alpha) \sum_{n=0}^{\infty}\left((2 n \pi+\phi)^{\alpha-2}+(2(n+1) \pi-\phi)^{\alpha-2}\right) \cos{\left( \frac{\alpha}{2} \pi \right)} \\
&>&0,
\end{eqnarray*}
and
\[\frac{\left|\operatorname{Im}(g_{\alpha}( \phi ))\right|}{\operatorname{Re}(g_{\alpha}( \phi ))}=\dfrac{\sum\limits_{n=0}^{\infty}\Big((2n\pi+\phi)^{\alpha-2}-(2(n+1)\pi-\phi)^{\alpha-2}\Big)}{\sum\limits_{n=0}^{\infty}\Big((2n\pi+\phi)^{\alpha-2}+(2(n+1)\pi-\phi)^{\alpha-2}\Big)}\tan{\left( \frac{\alpha}{2} \pi \right)}<\tan{\left( \frac{\alpha}{2} \pi \right)}.\]
Similarly, we have $\forall\phi\in[-\pi,0)$,
\begin{equation*}
\operatorname{Re}(g_{\alpha}( \phi ))=\operatorname{Re}(g_{\alpha}( -\phi ))>0
\end{equation*}
and
\[\frac{\left|\operatorname{Im}(g_{\alpha}( \phi ))\right|}{\operatorname{Re}(g_{\alpha}( \phi ))}=\dfrac{\left|-\textrm{Im}(g_{\alpha}(-\phi ))\right|}{\textrm{Re}(g_{\alpha}(-\phi ))}<\tan{\left( \frac{\alpha}{2} \pi \right)}.\]
Hence, by periodicity, we have $\forall \phi \in \mathbb{R} \backslash\{2 k \pi \mid k \in \mathbb{Z}\}$
\begin{equation*}
\operatorname{Re}(g_{\alpha}( \phi ))>0 \quad \text {and}\quad \frac{\left|\operatorname{Im}(g_{\alpha}( \phi ))\right|}{\operatorname{Re}(g_{\alpha}( \phi ))}<\tan{\left( \frac{\alpha}{2} \pi \right)}.
\end{equation*}
 \end{proof}

\begin{assumption}\label{assumption_1}
$\operatorname{Re}(w_{\boldsymbol{\beta}}(\boldsymbol{\theta}))$ is essentially positive and $\left|\operatorname{Im}(w_{\boldsymbol{\beta}}(\boldsymbol{\theta}))\right| \leq \mu_{\beta} \operatorname{Re}(w_{\boldsymbol{\beta}}(\boldsymbol{\theta}))$ with positive constant $\mu_{\beta}$ independent of matrix size $J$.
\end{assumption}
The above assumption can be easily met. It is easy to check that the discretization matrices for $\Delta$ and $\sum_{i=1}^d c_i \frac{\partial^{\beta_i}}{\partial\left|x_i\right|^{\beta_i}}$ are SPD with real and essentially positive generating function \cite{hon2024sine,pang2016fast,vong2019second,mazza2021symbol,li2023preconditioning,ding2023high}. In this case, $\left|\operatorname{Im}(w_{\boldsymbol{\beta}}(\boldsymbol{\theta}))\right|=0$. As for the operator $\sum_{i=1}^d \left(k_{i,+} \frac{\partial^{\beta_i}}{\partial_{+} x_i^{\beta_i}}+k_{i,-}\frac{\partial^{\beta_i}}{\partial_{-} x_i^{\beta_i}}\right)$, commonly used discretizations such as the shifted Gr{\"u}nwald-formula \cite{meerschaert2004finite}, and the weighted and shifted Gr{\"u}nwald-Letnikov difference (WSGD) formulas \cite{tian2015class,hao2015fourth}, can generate nonsymmetric Toeplitz discretization matrices. These matrices have generating functions that satisfy the above assumptions \cite{pang2016fast,lin2023tau,vong2019second,mazza2021symbol,tang2024fast}. 

\begin{proposition}\label{prop:main}
Let $f_{\alpha, \boldsymbol{\beta}}(\phi,\boldsymbol{\theta})$ be defined in \eqref{eqn:main_gen_function}. Then,
\[
\operatorname{ess\,sup}_{(\phi,\boldsymbol{\theta}) \in[-\pi, \pi]^{d+1}} \left|\frac{ \mathrm{Im}(f_{\alpha, \boldsymbol{\beta}}(\phi,\boldsymbol{\theta})) }{ \mathrm{Re}(f_{\alpha, \boldsymbol{\beta}}(\phi,\boldsymbol{\theta}))  }\right| \leq \eta,
\]
where
\begin{eqnarray*}
\eta= &=& \left\{
\begin{aligned}
&\tan \left( \frac{\alpha}{2} \pi\right),~~~~~\text{if}~G_J~\text{is~symmetric,}\\
&\max\left\{\mu_{\beta}, \tan \left( \frac{\alpha}{2} \pi\right)\right\}, ~~~~~\text{if}~G_J~\text{is nonsymmetric,}
\end{aligned}
\right.
\end{eqnarray*}
with $\mu_{\beta}$ defined in Assumption \ref{assumption_1}.

\end{proposition}

\begin{proof}

We have
  \begin{eqnarray*}
    \left|\frac{ \mathrm{Im}(f_{\alpha, \boldsymbol{\beta}}(\phi,\boldsymbol{\theta})) }{ \mathrm{Re}(f_{\alpha, \boldsymbol{\beta}}(\phi,\boldsymbol{\theta}))  }\right| &=& \frac{| \mathrm{Im}(w_{\boldsymbol{\beta}}(\boldsymbol{\theta}) + \kappa g_{\alpha}(\phi)  )  |}{ \mathrm{Re}(w_{\boldsymbol{\beta}}(\boldsymbol{\theta}) + \kappa g_{\alpha}(\phi))  } \\
  &\leq& \frac{| \mathrm{Im}(w_{\boldsymbol{\beta}}(\boldsymbol{\theta}) |+ \kappa | \mathrm{Im}(  g_{\alpha}(\phi)  )  |}{\mathrm{Re}(w_{\boldsymbol{\beta}}(\boldsymbol{\theta}))  + \kappa \mathrm{Re}( g_{\alpha}(\phi))  } \\
    &\leq&\left\{
    \begin{aligned}
    &\frac{| \mathrm{Im}(  g_{\alpha}(\phi)  )  |}{ \mathrm{Re}( g_{\alpha}(\phi))  },~~~~~\text{if}~G_J~\text{is~symmetric,}\\
    &\max\left\{\frac{| \mathrm{Im}(w_{\boldsymbol{\beta}}(\boldsymbol{\theta}) |}{\mathrm{Re}(w_{\boldsymbol{\beta}}(\boldsymbol{\theta}))  }, \frac{| \mathrm{Im}(  g_{\alpha}(\phi)  )  |}{ \mathrm{Re}( g_{\alpha}(\phi))  }\right\}, ~~~~~\text{if}~G_J~\text{is nonsymmetric,}
    \end{aligned}
    \right.\\
    &\leq&\left\{
    \begin{aligned}
    &\tan \left( \frac{\alpha}{2} \pi\right),~~~~~\text{if}~G_J~\text{is~symmetric,}\\
    &\max\left\{\mu_{\beta}, \tan \left( \frac{\alpha}{2} \pi\right)\right\}, ~~~~~\text{if}~G_J~\text{is nonsymmetric.}
    \end{aligned}
    \right.
    \end{eqnarray*}
    
\end{proof}

Combining Theorem \ref{conv_gmres_1} and Proposition \ref{prop:main}, we establish that $\hat P$, as defined in (\ref{Hermitian_pre}), is an ideal preconditioner. However, since it cannot be easily implemented in general, we now turn our attention to developing a practical preconditioner for matrix $A$ in \eqref{aaosystem}.

\subsection{Convergence analysis of the practical preconditioner $P$}\label{sec:modified_P}
Before defining the practical modified preconditioner $P$, we present some reasonable assumptions as follows.
\begin{assumption}\label{assumption_2}
For the multilevel ($d$-level) Toeplitz matrix $G_J$, there exists a fast diagonalizable SPD matrix $\tilde P$ such that the minimum eigenvalue of $\tilde P$ has a lower bound independent of matrix size and the spectrum of the preconditioned matrix $\tilde P^{-1} \mathcal{H}(G_J)$ is uniformly bounded, i.e., 
\begin{itemize}
  \item [\bf{(i)}] $\tilde P$ is SPD and $\inf \limits _{J>0} \lambda_{\min }(\tilde P) \geq \check{c}>0$.
  \item [\bf{(ii)}] $\tilde P$ is fast diagonalizable. In other words, for the orthogonal diagonalization of $\tilde P$, $\tilde P=\tilde S \Lambda_{\tilde P} \tilde S^{\top}$, both the orthogonal matrix $\tilde S$ and its transpose $\tilde S^{\top}$ have fast matrix-vector multiplications; the diagonal entries $\{\lambda_i\}_{i=1}^{J}$ of the diagonal matrix $\Lambda_{\tilde P}=\text{diag}(\lambda_i)_{i=1}^J$ are fast computable.
  \item [\bf{(iii)}] $\sigma({\tilde P}^{-1} \mathcal{H}(G_J)) \subset[\check{a}, \hat{a}]$ with $\check{a}$ and $\hat{a}$ being two positive constants independent of the matrix size parameter $J$.
\end{itemize}
\end{assumption}

The assumptions above can be easily satisfied. If the matrix $G_J$ arises from a class of low-order discretization schemes \cite{Meerschaert2006,SL1,CD2} of the multi-dimensional Riesz derivative $\mathcal{L}=\sum_{i=1}^d c_i \frac{\partial^{\alpha_i}}{\partial\left|x_i\right|^{\alpha_i}}$, then the well-known $\tau$ preconditioner \cite{Huang_2022} is a valid candidate of $\tilde P$, which is both fast diagonalizable (by the multi-dimensional sine transform matrix) and SPD with its minimum eigenvalue bounded below by a constant independent of the matrix size (referring to \cite{lin2024single}); moreover, the spectra of the preconditioned matrices lie in the open interval $(1/2,3/2)$ \cite{Huang_2022,zhang2023fast}. If $G_J$ arises from the high-order discretization \cite{ding2023high} of the Riesz derivative, then the $\tau$-matrix based preconditioner proposed in \cite{qu2024novel} is a viable option of $\tilde P$ with fast matrix-vector multiplications and $\sigma({\tilde P}^{-1} G_J) \subset (3/8,2)$ \cite{qu2024novel} and the lower bound of minimum eigenvalue of $\tilde P$ also has been proved in \cite{huang2024efficient}. If the matrix $G_J$ arises from the central difference discretization of the constant Laplacian operator $\mathcal{L}=\Delta$, then the discretization matrix $L_1$ of $-\Delta$ is a suitable choice of $\tilde P$ since $L_1$ itself is a $\tau$-matrix and meets Assumption \ref{assumption_2}$~{\bf(i)}$-${\bf (iii)}$, see for example \cite{lin2021parallel,lin2024matching}. If $G_J$ arises from the shifted Gr{\"u}nwald-formula \cite{meerschaert2004finite} or the WSGD formulas \cite{tian2015class,hao2015fourth} for $\sum_{i=1}^d \left(k_{i,+} \frac{\partial^{\beta_i}}{\partial_{+} x_i^{\beta_i}}+k_{i,-}\frac{\partial^{\beta_i}}{\partial_{-} x_i^{\beta_i}}\right)$, the $\tau$-preconditioner based on $\mathcal{H}(G_J)$ also satisfies all the assumptions above, see for example \cite{Huang_2022,lin2023tau,huang2022preconditioners,tang2024fast}. 

With a preconditioner $\tilde P$ that satisfies Assumption \ref{assumption_2}, a practical novel preconditioner $P$ for $A$ in (\ref{aaosystem}) can be defined as
\begin{equation}\label{practical_pre}
P = \tilde P \otimes I_N + I_J \otimes \rho \tau(\mathcal{H}(B^{(\alpha)}_N)).
\end{equation}
From \eqref{taumatdiag} \& \eqref{sigmicomp} and properties of the one-dimensional sine transform matrix $S_m$, we know that $P$ can be diagonalized in the following form
\begin{align*}
        P=S \Lambda S,\label{btaudiag} \quad S:=\tilde S\otimes S_{N}^{\top},\quad \Lambda:= \Lambda_{\tilde P}\otimes I_N + I_J \otimes \rho \Upsilon,
\end{align*}
where $\Upsilon$ contains the eigenvalues of $\tau(\mathcal{H}(B^{(\alpha)}_N))$, which, in combination with Assumption \ref{assumption_2}~{\bf (ii)}, illustrates that the product of $P^{-1}$ and a given vector can be efficiently computed.

Since $\tilde P$ is SPD by assumption, to prove the positive definiteness of $P$, it suffices to show that $\tau(\mathcal{H}(B^{(\alpha)}_N))$ is SPD. Before doing so, we first present some properties of of $a_k$ related to the $L1$ formula for the Caputo derivative.

\begin{lemma} \cite{zhang2011alternating,chen2009finite}\label{prop_L1_coef}
Let $\alpha \in (0,1)$, $a_k=(k+1)^{1-\alpha}-k^{1-\alpha}$, $k=0,1,\ldots$. Then,
\begin{itemize}
  \item [\bf{(i)}] $1=a_0>a_1>\cdots>a_N>\cdots \rightarrow 0$; 
  \item [\bf{(ii)}] $a_0+\sum_{k=1}^{N-1}\left(a_{k}-a_{k-1}\right)=a_{N-1}$.
\end{itemize}
\end{lemma} 

\begin{lemma}\cite{lin2018separable}\label{positive_Hermitian_L1}
For any $\alpha \in (0,1)$, it holds that $\mathcal{H}(B^{(\alpha)}_N)$ is SPD. 
\end{lemma}

The following lemma guarantees the positive definiteness of $P$.

\begin{lemma}\label{lemma:PS_SPD}
The matrix $P$ defined in \eqref{assumption_1} is SPD and $\inf\limits_{J>0} \lambda_{\min }\left(P\right) \geq \check{c}>0$ with $\check{c}>0$ defined in Assumption \ref{assumption_2}$~{\bf(i)}$.
\end{lemma}
\begin{proof}
Since $\tilde P$ is SPD from Assumption \ref{assumption_2}$~{\bf(i)}$ and the symmetry of $\tau(\mathcal{H}(B^{(\alpha)}_N))$ is straightforward, it suffices to prove the positive definiteness of $\tau(\mathcal{H}(B^{(\alpha)}_N))$. From (\ref{sigmicomp}), (\ref{eqn:matrix_time_fractional}) and Lemma \ref{prop_L1_coef}, we know that
\begin{eqnarray}\nonumber
\lambda_k(\tau(\mathcal{H}(B^{(\alpha)}_N)))&=&l_{0}^{(\alpha)}+ \sum_{j=1}^{N-1} l_{j}^{(\alpha)} \cos \left(\frac{jk\pi}{N+1}\right)\\\nonumber
&\geq&l_{0}^{(\alpha)}+ \sum_{j=1}^{N-1} l_{j}^{(\alpha)}\\\nonumber
&=&a_0+\sum_{k=1}^{N-1}\left(a_{k}-a_{k-1}\right) \\\nonumber
&=&a_{N-1},
\end{eqnarray}
which means $\tau(\mathcal{H}(B^{(\alpha)}_N))$ is SPD.

Thus, knowing that
\begin{equation*}
P = \tilde P \otimes I_N + I_J \otimes \rho \tau(\mathcal{H}(B^{(\alpha)}_N)), 
\end{equation*}
we conclude $P$ is SPD and
\begin{equation*}
\inf\limits_{J>0} \lambda_{\min }\left(P\right) \geq \inf\limits_{J>0} \lambda_{\min }\left(\tilde P\right) \geq \check{c}>0.
\end{equation*}
\end{proof}

\begin{lemma} \label{M_matrix}
The matrix $\tau(\mathcal{H}(B^{(\alpha)}_N))$ is strictly diagonally dominant with positive diagonal entries and negative off-diagonal entries.
\end{lemma}
\begin{proof}
Denote $p_{i,j}$ be the $(i,j)$-th entry of $\tau(\mathcal{H}(B^{(\alpha)}_N))$. Based on the monotonicity of $\{a_k\}_{k=0}^{N-1}$ from Lemma \ref{prop_L1_coef} and the definition of $\tau$-matrix, it is easy to check that $p_{ij}\leq 0$ for $i\neq j$.
In addition, the diagonal entries of $\tau(\mathcal{H}(B^{(\alpha)}_N))$ contain
    $\{l_0^{(\alpha)}-\frac{1}{2}l_2^{(\alpha)}, l_0^{(\alpha)}-\frac{1}{2}l_4^{(\alpha)}, \ldots, l_0^{(\alpha)}-\frac{1}{2}l_{N-2}^{(\alpha)}, l_0^{(\alpha)}\}$.
    By Lemma \ref{prop_L1_coef} again, we know that for $k\geq1$, $l_k^{(\alpha)}=a_k-a_{k-1}<0$ and $l_0^{(\alpha)}-\frac{1}{2}l_k^{(\alpha)}>l_0^{(\alpha)}$, which means $p_{ii}\geq l_0^{(\alpha)}> 0$ for each $i=1,2,\cdots,N$. On the other hand,

\begin{equation*}
\begin{aligned}
\sum_{j=1, j \neq i}^N\left|p_{i j}\right| &\leq \sum_{j=1}^{i-1}\left|\frac{1}{2}l_j^{(\alpha)}\right|+\sum_{j=1}^{N-i}\left|\frac{1}{2}l_j^{(\alpha)}\right|+\left(\sum_{j=i+1}^{N-1}\left|\frac{1}{2}l_j^{(\alpha)}\right|+\sum_{j=N-i+2}^{N-1}\left|\frac{1}{2}l_j^{(\alpha)}\right|\right) \\
&<\sum_{j=1}^{N-1}\left|\frac{1}{2}l_j^{(\alpha)}\right|+\sum_{j=1}^{N-1}\left|\frac{1}{2}l_j^{(\alpha)}\right| \\
&=\sum_{j=1}^{N-1}\left|l_j^{(\alpha)}\right| \\
&=-(l_1^{(\alpha)}+l_2^{(\alpha)}+\cdots+l_{N-1}^{(\alpha)})\\
&=a_0-a_{N-1}\\
&< l_0^{(\alpha)}\\
&\leq p_{ii},
\end{aligned}
\end{equation*}
which means  $\tau(\mathcal{H}(B^{(\alpha)}_m))$ is diagonally dominant.
\end{proof}

\begin{lemma}\label{temporal_bound}
The eigenvalues of $\tau(\mathcal{H}(B^{(\alpha)}_N))^{-1}  \mathcal{H}(B^{(\alpha)}_N)  $ lie in $(1/2, 3/2)$ for $\alpha \in (0, 1)$ and $N>0$.
\end{lemma}
\begin{proof}
Considering the following matrix decomposition \[\tau(\mathcal{H}(B^{(\alpha)}_N))^{-1}  \mathcal{H}(B^{(\alpha)}_N) = \mathcal{I} + \tau(\mathcal{H}(B^{(\alpha)}_N))^{-1} H(\mathcal{H}(B^{(\alpha)}_N)),\] it suffices to show the spectral distribution of $\tau(\mathcal{H}(B^{(\alpha)}_N))^{-1}  H(\mathcal{H}(B^{(\alpha)}_N))$. 

Let $h_{i,j}$ be the $(i,j)$-th entry of $H(\mathcal{H}(B^{(\alpha)}_N))$ and $(\lambda,\tilde x)$ with $\tilde x=\left[\tilde x_1, \tilde x_2, \ldots, \tilde x_{N}\right]^{\top}$ and $\max\limits_{1 \leq j \leq {N}}\left|\tilde x_j\right|=\left|\tilde x_k\right|=1$ be an eigenpair of $\tau(\mathcal{H}(B^{(\alpha)}_N))^{-1}  H(\mathcal{H}(B^{(\alpha)}_N))$. Then, we have
\begin{equation*}
\sum_{j=1}^{N} h_{k j} \tilde x_j=\lambda \sum_{j=1}^{N} p_{k j} \tilde x_j,
\end{equation*}
which is equivalent to
\begin{equation*}
\lambda p_{k k} \tilde x_k=\sum_{j=1}^{N} h_{k j} \tilde x_j-\lambda \sum_{j=1, j \neq k}^{N} p_{i j} \tilde x_j.
\end{equation*}
By taking the absolute value on both sides of the above equation, the following inequality holds
\begin{equation*}
|\lambda|\left|p_{k k}\right| \leq \sum_{j=1}^{N}\left|h_{k j}\right|+|\lambda| \sum_{j=1, j \neq k}^{N}\left|p_{k j}\right|.
\end{equation*}
Since $\tau(\mathcal{H}(B^{(\alpha)}_N))$ is diagonally dominant from Lemma \ref{M_matrix}, the following inequality holds
\begin{equation*}
|\lambda| \leq \frac{\sum\limits_{j=1}^{N}\left|h_{k j}\right|}{\left|p_{k k}\right|-\sum\limits_{j=1, j \neq k}^{N}\left|p_{k j}\right|}.
\end{equation*}
By noting that $p_{ii}> 0$; $p_{ij}\leq 0$ for $i\neq j$ and $h_{i j}\leq 0$, it follows that 
\begin{equation*}
\begin{aligned}
&\left|p_{k k}\right|-\sum_{j=1, j \neq k}^{N}\left|p_{k j}\right|-2 \sum_{j=1}^{N}\left|h_{k j}\right| \\
&=\left(l_0^{(\alpha)}-h_{k k}\right)-\sum_{j=1, j \neq k}^{N}\left(h_{k j}-\frac{1}{2}l_{|k-j|}^{(\alpha)}\right)+2 \sum_{j=1}^{N} h_{k j}\\ 
&= l_0^{(\alpha)}+\sum_{j=1, j \neq k}^{N} \frac{1}{2}l_{|k-j|}^{(\gamma)}+\sum_{j=1}^{N} h_{k j} \\
&= l_0^{(\alpha)}+\frac{1}{2}\left(\sum_{j=1}^{k-1} l_j^{(\alpha)}+\sum_{j=1}^{N-k} l_j^{(\alpha)}\right)+\frac{1}{2}\left(\sum_{j=k+1}^{N-1} l_j^{(\alpha)}+\sum_{j=N-k+2}^{N-1} l_j^{(\alpha)}\right) \\
&\geq  l_0^{(\alpha)}+\sum_{j=1}^{N-1} l_j^{(\alpha)} \\
&=a_{N-1}\\
&>0,
\end{aligned}
\end{equation*}
which indicates that the spectrum of $\tau(\mathcal{H}(B^{(\alpha)}_N))^{-1}  H(\mathcal{H}(B^{(\alpha)}_N))$ is uniformly bounded by $-1/2$ and $1/2$. The proof is complete.
\end{proof}

The following lemma shows that $P$ and $\mathcal{H}( A)$ are spectrally equivalent.

\begin{lemma}\label{prop:eigen_S}
Let $A, P$ be the matrices defined in \eqref{aaosystem} and \eqref{practical_pre}, respectively. Then, the eigenvalues of $P^{-1/2} \mathcal{H}(A) P^{-1/2}$ lie in $(\check{b}, \hat{b})$ with $\check{b}=\min{\left\{\check{a},\frac{1}{2}\right\} }$ and $ \hat{b}=\max{\left\{ \hat{a},\frac{3}{2}\right\} }$.
\end{lemma}
\begin{proof}
Let $(\lambda, \mathbf{w})$ be an arbitrary eigenpair of $P^{-1}   \mathcal{H}(A)$. Then, it holds
\begin{eqnarray*}
    \lambda &=& \frac{ \mathbf{w}^*    \mathcal{H}( A) \mathbf{w} }{\mathbf{w}^* P \mathbf{w}} \\
    &=&\frac{ \mathbf{w}^*    \left( \mathcal{H}(G_J)\otimes I_N + I_J \otimes \kappa \mathcal{H}(B^{(\alpha)}_N)  \right) \mathbf{w} }{ \mathbf{w}^*    \left( \tilde P\otimes I_N + I_J \otimes \kappa \tau(\mathcal{H}(B^{(\alpha)}_N))  \right) \mathbf{w} }.
\end{eqnarray*}
Now, combining Assumption \ref{assumption_2}$~{\bf (iii)}$ with the Rayleigh quotient theorem, we have
\begin{eqnarray*}\nonumber
    \check{a} &\leq&\lambda_{\min} \left( \tilde P^{-1} \mathcal{H}(G_J) \right)  \\\nonumber
    &\leq& 
\frac{ \mathbf{y_1}^*    \mathcal{H}(G)  \mathbf{y_1} }{\mathbf{y_1}^* \tilde P \mathbf{y_1}} \\\label{eqn:tau_W}
&\leq&  \lambda_{\max} \left( \tilde P^{-1} \mathcal{H}(G_J) \right) \leq \hat{a}, 
\end{eqnarray*}
for any nonzero vector $\mathbf{y_1}$.

Also, based on Lemma \ref{temporal_bound}, we know that
\begin{eqnarray*}
    \frac{1}{2} &\leq&\lambda_{\min} \left( \tau \left(\mathcal{H}\left(B^{(\alpha)}_N\right) \right)^{-1}\mathcal{H}\left(B^{(\alpha)}_N\right)   \right)  \\
    &\leq& 
\frac{ \mathbf{y_2}^*   \mathcal{H}\left(B^{(\alpha)}_N\right) \mathbf{y_2} }{\mathbf{y_2}^* \tau \left(\mathcal{H}\left(B^{(\alpha)}_N\right) \right) \mathbf{y_2}} \\
&\leq&  \lambda_{\max} \left(  \tau \left(\mathcal{H}\left(B^{(\alpha)}_N\right) \right)^{-1}\mathcal{H}\left(B^{(\alpha)}_N\right)  \right) \leq \frac{3}{2},
\end{eqnarray*}
for any nonzero vector $\mathbf{y_2}$.

Then, we have
\begin{eqnarray*}
    \check{a} &=& \check{a} \cdot \frac{ \mathbf{w}^*  \left( \tilde P \otimes I_N  \right)   \mathbf{w} }{ \mathbf{w}^*  \left( \tilde P \otimes I_N  \right)  \mathbf{w} } \\
    &\leq& \frac{ \mathbf{w}^*    \left( \mathcal{H}(G_J) \otimes I_N  \right) \mathbf{w} }{ \mathbf{w}^*    \left( \tilde P \otimes I_N   \right) \mathbf{w} } \\
    &\leq& \hat{a} \cdot \frac{ \mathbf{w}^*  \left( \tilde P \otimes I_N  \right)   \mathbf{w} }{ \mathbf{w}^*  \left( \tilde P \otimes I_N  \right)  \mathbf{w} } = \hat{a},
\end{eqnarray*}
and
    \begin{eqnarray*}
    \frac{1}{2} &=& \frac{1}{2} \cdot \frac{ \mathbf{w}^*  \left( I_J \otimes \tau(\mathcal{H}(B^{(\alpha)}_N)) \right)   \mathbf{w} }{ \mathbf{w}^*  \left( I_J \otimes \tau(\mathcal{H}(B^{(\alpha)}_N))  \right)  \mathbf{w} } \\
    &\leq& \frac{ \mathbf{w}^*    \left( I_J \otimes \mathcal{H}(B^{(\alpha)}_N)  \right) \mathbf{w} }{ \mathbf{w}^*    \left( I_J \otimes \tau(\mathcal{H}(B^{(\alpha)}_N))   \right) \mathbf{w} } \\
    &\leq& \frac{3}{2} \cdot \frac{ \mathbf{w}^*    \left( I_J \otimes \tau(\mathcal{H}(B^{(\alpha)}_N))   \right) \mathbf{w} }{ \mathbf{w}^*   \left( I_J \otimes \tau(\mathcal{H}(B^{(\alpha)}_N))   \right) \mathbf{w} } = \frac{3}{2}.
\end{eqnarray*}

By Lemma \ref{lemma:wghtsumbdlem}, it follows that
\begin{equation*}
    \min{\left\{\check{a},\frac{1}{2}\right\} }\leq \frac{ \mathbf{w}^*    \left( \mathcal{H}(G_J) \otimes I_N + I_J \otimes \kappa \mathcal{H}(B^{(\alpha)}_N)  \right) \mathbf{w} }{ \mathbf{w}^*    \left(\tilde P \otimes I_N + I_J \otimes \kappa \tau(\mathcal{H}(B^{(\alpha)}_N))  \right) \mathbf{w} }\\
    \leq\max{\left\{ \hat{a},\frac{3}{2}\right\} }, 
\end{equation*}
which implies $\lambda \in (\check{b}, \hat{b})$. The proof is complete.
\end{proof}

\begin{lemma}\label{lemma:eigen_Skew}
Let $A, P$ be the matrices defined in \eqref{aaosystem} and \eqref{practical_pre}, respectively. Then,
\begin{equation*}
    \rho\left( \mathcal{S}\left( P^{-1/2} A P^{-1/2} \right)\right) \leq \varsigma,    
\end{equation*}
where
\begin{eqnarray*}
        \varsigma &=& \left\{
    \begin{aligned}
    &\frac{3}{2}\tan \left( \frac{\alpha}{2} \pi\right),~~~~~\text{if}~G~\text{is~symmetric},\\
    &\max\left\{\mu_{\beta} \hat{a}, \frac{3}{2}\tan \left( \frac{\alpha}{2} \pi\right)\right\}, ~~~~~\text{if}~G~\text{is nonsymmetric,}
    \end{aligned}
    \right.
    \end{eqnarray*}
with $\mu_{\beta}$ and $\hat{a}$ defined in Assumption \ref{assumption_2}$~{\bf(i)}$.

\end{lemma}
\begin{proof}
     Let $(\lambda, \mathbf{w})$ be an arbitrary eigenpair of $P^{-1}   \mathcal{S}( A)$. Then, it holds
\begin{eqnarray*}
        |\lambda| &=& \left| \frac{ \mathbf{w}^*    \mathcal{S}( A) \mathbf{w} }{\mathbf{w}^* P \mathbf{w}} \right|\\
        &=&\frac{\left| \mathbf{w}^*    \left( \mathcal{S}(G_J) \otimes I_N + I_J \otimes \kappa \mathcal{S}(B^{(\alpha)}_N)  \right) \mathbf{w} \right|}{ \mathbf{w}^*    \left( \tilde P \otimes I_N + I_J \otimes \kappa \tau(\mathcal{H}(B^{(\alpha)}_N))  \right) \mathbf{w} }\\
        &\leq&\frac{\left| \mathbf{w}^*    \left( \mathcal{S}(G_J) \otimes I_N \right) \mathbf{w} \right| + \left| \mathbf{w}^* \left( I_J \otimes \kappa \mathcal{S}(B^{(\alpha)}_N)  \right) \mathbf{w} \right|}{ \mathbf{w}^*    \left( \tilde P \otimes I_N + I_J \otimes \kappa \tau(\mathcal{H}(B^{(\alpha)}_N))  \right) \mathbf{w} }\\
        &\leq& \left\{
    \begin{aligned}
    &\frac{\left| \mathbf{w}^*    \left( I_J \otimes \mathcal{S}(B^{(\alpha)}_N)  \right) \mathbf{w} \right|}{ \mathbf{w}^*    \left( I_J \otimes \tau(\mathcal{H}(B^{(\alpha)}_N))  \right) \mathbf{w} },~~~~~\text{if}~G_J~\text{is~symmetric},\\
    &\max\left\{\frac{\left| \mathbf{w}^*    \left(\mathcal{S}(G_J) \otimes I_N  \right) \mathbf{w} \right|}{ \mathbf{w}^*    \left( \tilde P \otimes I_N  \right) \mathbf{w} }, \frac{\left| \mathbf{w}^*    \left( I_J \otimes \mathcal{S}(B^{(\alpha)}_N)  \right) \mathbf{w} \right|}{ \mathbf{w}^*    \left( I_J \otimes \tau(\mathcal{H}(B^{(\alpha)}_N))  \right) \mathbf{w} }\right\}, ~~~~~\text{if}~G_J~\text{is~nonsymmetric.}
    \end{aligned}
    \right.
    \end{eqnarray*}
By noting that $\left|\operatorname{Im}(w_{\boldsymbol{\beta}}(\boldsymbol{\theta}))\right| \leq \mu_{\beta} \operatorname{Re}(w_{\boldsymbol{\beta}}(\boldsymbol{\theta}))$ and $\left|\operatorname{Im}(g_{\alpha}( \phi ))\right|<\tan{\left( \frac{\alpha}{2} \pi \right)}\operatorname{Re}(g_{\alpha}( \phi ))$, we have \cite{lin2023tau}
\begin{equation*}
\left| \mathbf{y}^*\mathcal{S}\left(G_J\right) \mathbf{y} \right|  \leq \mu_{\beta} \mathbf{y}^*   \mathcal{H}\left(G_J\right) \mathbf{y} ~~~~~~(\text{if}~G_J~\text{is nonsymmetric})
\end{equation*}
and
\begin{equation*}
\left| \mathbf{y}^*\mathcal{S}\left(B^{(\alpha)}_N\right) \mathbf{y} \right|  \leq \tan \left( \frac{\alpha}{2} \pi\right) \mathbf{y}^*   \mathcal{H}\left(B^{(\alpha)}_N\right) \mathbf{y}.
\end{equation*}
Thus,
\begin{equation*}
        \frac{ \left| \mathbf{y}^*   \mathcal{S}\left(G_J\right) \mathbf{y} \right|}{\mathbf{y}^* \tilde P \mathbf{y}} = \frac{\left| \mathbf{y}^*   \mathcal{S}\left(G_J\right) \mathbf{y} \right| }{\mathbf{y}^* \mathcal{H}\left(G_J\right)  \mathbf{y}} \cdot \frac{\mathbf{y}^* \mathcal{H}\left(G_J\right)  \mathbf{y}}{\mathbf{y}^* \tilde P \mathbf{y}} \leq \mu_{\beta} \hat{a}
    \end{equation*} 
and
    \begin{equation*}
        \frac{ \left| \mathbf{y}^*   \mathcal{S}\left(B^{(\alpha)}_N\right) \mathbf{y} \right|}{\mathbf{y}^* \tau \left(\mathcal{H}\left(B^{(\alpha)}_N\right) \right) \mathbf{y}} = \frac{\left| \mathbf{y}^*   \mathcal{S}\left(B^{(\alpha)}_N\right) \mathbf{y} \right| }{\mathbf{y}^* \mathcal{H}\left(B^{(\alpha)}_N\right)  \mathbf{y}} \cdot \frac{\mathbf{y}^* \mathcal{H}\left(B^{(\alpha)}_N\right)  \mathbf{y}}{\mathbf{y}^* \tau \left(\mathcal{H}\left(B^{(\alpha)}_N\right) \right) \mathbf{y}} \leq \frac{3}{2}\tan \left( \frac{\alpha}{2} \pi\right)
    \end{equation*}
for any nonzero vector $\mathbf{y}$.

Therefore, we have
\begin{eqnarray*}
\rho\left( \mathcal{S}\left( P^{-1/2} A P^{-1/2} \right)\right) \leq \max\left\{\mu_{\beta} \hat{a}, \frac{3}{2}\tan \left( \frac{\alpha}{2} \pi\right)\right\}.
\end{eqnarray*}
Namely, 
\begin{eqnarray*}
    \rho\left( \mathcal{S}\left( P^{-1/2} A P^{-1/2} \right)\right)&\leq& \left\{
\begin{aligned}
&\frac{3}{2}\tan \left( \frac{\alpha}{2} \pi\right),~~~~~\text{if}~G_J~\text{is~symmetric},\\
&\max\left\{\mu_{\beta} \hat{a}, \frac{3}{2}\tan \left( \frac{\alpha}{2} \pi\right)\right\}, ~~~~~\text{if}~G_J~\text{is nonsymmetric.}
\end{aligned}
\right.
\end{eqnarray*}
The proof is complete.
\end{proof}

\begin{theorem}\label{theorem_main}
    Let $A, P$ be the matrices defined in \eqref{aaosystem} and \eqref{practical_pre}, respectively. Then, the residuals of the iterates generated by applying (restarted or non-restarted) GMRES with an arbitrary initial guess to solve $P^{-\frac{1}{2}} A P^{-\frac{1}{2}}\mathbf{v}=P^{-\frac{1}{2}}\mathbf{w}$ satisfy
    \[
       \| \mathbf{r}_k \|_2 \leq  \omega^{k} \| \mathbf{r}_0 \|_2,
    \]
    where $\mathbf{r}_k = P^{-\frac{1}{2}}\mathbf{w} - P^{-\frac{1}{2}} A P^{-\frac{1}{2}}\mathbf{v}_k$ is the residual vector at the $k$-th GMRES iteration with $\mathbf{v}_k$ ($k \geq 1$) being the corresponding iterative solution, and $\omega$ is a constant independent of $N$ and $J$ defined as follows
        $$
        \omega := \sqrt{\frac{2+4\varsigma^2}{3+4\varsigma^2}} \in \left[ \sqrt{\frac{2}{3}},1  \right) \subset (0,1),
        $$
        with $\varsigma$ defined in Lemma \ref{lemma:eigen_Skew}.
\end{theorem}
\begin{proof}
First of all, since 
\begin{eqnarray*}
    \mathcal{H} \left( P^{-\frac{1}{2}} A P^{-\frac{1}{2}} \right)
    &=&  P^{-\frac{1}{2}} \mathcal{H}(A) P^{-\frac{1}{2}} \\
    &\succ& \mathcal{O},
\end{eqnarray*}
     Lemma \ref{lemma:gmres} can be used.

Also, $\mathcal{S}\left( P^{-1/2} A P^{-1/2} \right)=P^{-\frac{1}{2}} \mathcal{S} (A) P^{-\frac{1}{2}}$, 
we know by Lemma \ref{lemma:eigen_Skew} that
\begin{eqnarray*}
\rho\left( \mathcal{S}\left( P^{-1/2} A P^{-1/2} \right)\right) \leq \varsigma.
\end{eqnarray*}

Thus, combining Lemmas \ref{lemma:gmres}, \ref{prop:eigen_S} and \ref{lemma:eigen_Skew}, the residuals of the iterates generated by applying (restarted or non-restarted) GMRES with an arbitrary initial guess to solve $P^{-\frac{1}{2}} A P^{-\frac{1}{2}} \mathbf{v}=P^{-\frac{1}{2}}\mathbf{w}$ satisfy
\begin{eqnarray*}
\| \mathbf{r}_k \|_2 &\leq& \left( \sqrt{1 - \left(\frac{ (\frac{1}{2})^2}{ (\frac{1}{2})(\frac{3}{2}) + \varsigma^2} \right)}     \right)^{k} \| \mathbf{r}_0 \|_2\\
&=& \left(\sqrt{\frac{2+4\varsigma^2}{3+4\varsigma^2}}  \right)^{k} \| \mathbf{r}_0 \|_2.
\end{eqnarray*}    
\end{proof}

With Lemma \ref{lemma:PS_SPD}, a direct application of Lemma \ref{residual_relation} using the practical preconditioner $P$, as defined in (\ref{practical_pre}), yields the following theorem. This theorem guarantees that GMRES with $P$ can achieve mesh-independent convergence for $A$.

\begin{theorem}
Let $\hat{\mathbf{u}}_0$ be the initial guess for (\ref{two_sided_sys}) and $\mathbf{u}_0 := P^{-1/2}\hat{\mathbf{u}}_0$ be the initial guess for (\ref{one_sided_sys}). Let $\mathbf{u}_j$ ($\hat{\mathbf{u}}_j$, respectively) be the $j$-th $(j\geq1)$ iteration solution derived by applying GMRES solver to (\ref{one_sided_sys}) ((\ref{two_sided_sys}), respectively) with $\mathbf{u}_0$ ($\hat{\mathbf{u}}_0$, respectively) as their initial guess. Then,
\begin{equation*}
\left\|\mathbf{r}_j\right\|_2 \leq \frac{1}{\sqrt{\check{c}}}\left\|\hat{\mathbf{r}}_j\right\|_2
\end{equation*}
where $\mathbf{r}_j:=P^{-1} \mathbf{f} - P^{-1} A\mathbf{u}_j$ ($\hat{\mathbf{r}}_j := P^{-1/2} \mathbf{f} - P^{-1/2} AP^{-1/2} \hat{\mathbf{u}}_j$, respectively) denotes the residual vector at the $j$-th GMRES iteration for (\ref{one_sided_sys}) ((\ref{two_sided_sys}), respectively) and $\check{c}$ defined in Assumption \ref{assumption_2}$~{\bf(i)}$ is a constant independent of $N$ and $J$. 
\end{theorem}

\begin{remark}\label{Remark_variable_Laplacian}
It is worth noting that if $\mathcal{L}$ in Problem (\ref{main_problem}) is the variable coefficient Laplacian operator $\nabla \cdot(a(\boldsymbol{x})\nabla)$, then the discretization matrix $L_1$ from the constant Laplacian $-\Delta$ is a suitable choice for $\tilde P$ in (\ref{practical_pre}), provided that Assumption \ref{assumption_2} is satisfied; see \cite{lin2021parallel,lin2024matching} for more detail. Our convergence analysis in Subsection \ref{sec:modified_P} is still applicable, even though in this case the corresponding all-at-once matrix $A$ in (\ref{aaosystem}) is no longer a multilevel Toeplitz matrix.
\end{remark}

\section{Numerical experiments}\label{sec:numeric}

In this section, some numerical results are presented to verify the numerical accuracy and efficiency of the proposed preconditioner. All numerical experiments are performed using MATLAB 2021a on a HP Z620 workstation equipped with dual Xeon E5-2690 v2 10-Cores 3.0GHz CPUs, 128GB RAM running Ubuntu 20.04 LTS.

In the tables, `CPU(s)' represents the CPU time in seconds for solving the system (\ref{aaosystem});  `Iter' stands for the iteration numbers of different methods; and denotes by 'Error', the error between the numerical solution and the exact solution under the discrete maximum norm, i.e.,
\begin{equation*}
\textrm{Error}=\left\|u^*-\mathbf{u}\right\|_{\infty},
\end{equation*}
where $u^*$ is the exact solution. When CPU time is large than 3000 seconds, we stop the iteration by hands and denote the results in tables as `-'.

In the implementations of GMRES and PGMRES methods, we adopt the Matlab built-in function $\mathbf{gmres}$ with $\mathit{restart}=20$ and $\mathit{maxit}=1000$. The initial guess of GMRES methods (including PGMRES method) at each time step is chosen as the zero vector, and the stopping criterion is set as
$$
\frac{\|\mathbf{r}_k\|_2}{\|\mathbf{r}_0\|_2}\leq 10^{-8},
$$
where $r_k$ denotes the residual vector at the $k$-th iteration. Furthermore, to reduce the operation costs, all the matrix-vector multiplications in GMRES and PGMRES methods are fast evaluated via the MATLAB built-in functions {\bf fft}, {\bf ifft} and {\bf dst}.

\begin{example} \label{ex_1}
{\rm

Consider the problem (\ref{main_problem}) with $\mathcal{L}=\Delta$, $\Omega=(0,1)^2$, $T=1$ and the source term 
\begin{equation*}
\begin{aligned}
f(x_1, x_2, t)= & \frac{6 t^{3-\alpha}}{\Gamma(4-\alpha)} x_1^3 x_2^3(1-x_1)^2(1-x_2)^2 \\
& -t^3\left[x_2^3(1-x_2)^2\left(20 x_1^3-24 x_1^2+6 x_1\right)+x_1^3(1-x_1)^2\left(20 x_2^3-24 x_2^2+6 x_2\right)\right] .
\end{aligned}
\end{equation*}
The exact solution $u(x_1,x_2,t)=t^3 x_1^3 x_2^3(1-x_1)^2(1-x_2)^2$.

}
\end{example}

\begin{table}[t]%
\centering
\tabcolsep=5pt
    \caption{Results of GMRES method with and without preconditioner for Example \ref{ex_1} with $\mu=\frac{1}{2^{8}}$.}
    \label{table1}
    \begin{tabular}{cccccccc}
  \toprule
   \multirow{2}{*}{$\alpha$}  & \multirow{2}{*}{$h$}  &  \multicolumn{2}{c}{GMRES} &  \multicolumn{3}{c}{PGMRES}\\  \cmidrule(r){3-4} \cmidrule(r){5-7} 
     &      &CPU(s)  &Iter &Error &CPU(s)  &Iter \\
\hline
  \multirow{4}{*}{0.2}
         &$\frac{1}{32}$    &7.62        &244   &5.3880e-6   &2.29  &5      \\
         &$\frac{1}{64}$    &77.44       &768   &1.3520e-6   &3.36  &5       \\
         &$\frac{1}{128}$   &1192.54     &2832  &3.3875e-7   &9.24  &5        \\
         &$\frac{1}{256}$   &-           &-     &8.5437e-8   &26.72 &5        \\
\hline
  \multirow{4}{*}{0.5}                                                                                 &$\frac{1}{32}$    &7.33        &239   &5.3067e-6   &4.11  &10      \\
         &$\frac{1}{64}$    &75.61       &748   &1.3397e-6   &6.17  &10       \\
         &$\frac{1}{128}$   &1134.25     &2745  &3.4382e-7   &16.61 &10        \\
         &$\frac{1}{256}$   &-           &-     &9.4982e-8   &48.29 &10        \\    
\hline
  \multirow{4}{*}{0.8}                            
         &$\frac{1}{32}$    &7.13        &230   &5.2821e-6   &9.00   &21      \\
         &$\frac{1}{64}$    &70.89       &710   &1.4028e-6   &13.54  &21       \\
         &$\frac{1}{128}$   &1060.29     &2594  &4.3152e-7   &36.21  &21        \\
         &$\frac{1}{256}$   &-           &-     &1.9424e-7   &106.33 &21        \\
\bottomrule
    \end{tabular}
\end{table}

\begin{table}[t]%
\centering
\tabcolsep=5pt
    \caption{Results of GMRES method with and without preconditioner for Example \ref{ex_1} with $h=h_1=h_2=\frac{1}{2^{8}}$.}
    \label{table2}
    \begin{tabular}{ccccccc}
  \toprule
   \multirow{2}{*}{$\alpha$}  & \multirow{2}{*}{$\mu$}  &  \multicolumn{2}{c}{GMRES} &  \multicolumn{3}{c}{PGMRES}\\  \cmidrule(r){3-4} \cmidrule(r){5-7} 
     &     &CPU(s)  &Iter &Error &CPU(s)  &Iter \\
\hline
  \multirow{4}{*}{0.2}
         &$\frac{1}{8}$     &457.98    &11007   &4.9161e-7   &0.64  &4      \\
         &$\frac{1}{16}$    &887.62    &10998   &2.0965e-7   &1.42  &5       \\
         &$\frac{1}{32}$    &1969.35   &10995   &1.2170e-7   &3.12  &5        \\
         &$\frac{1}{64}$    &-         &-   &9.5547e-8   &6.83  &5        \\
\hline
  \multirow{4}{*}{0.5}                                                                                               
         &$\frac{1}{8}$     &437.65    &10717   &2.2444e-6   &0.74  &6      \\
         &$\frac{1}{16}$    &835.34    &10678   &8.8390e-7   &1.85  &7       \\
         &$\frac{1}{32}$    &1923.86   &10659   &3.6994e-7   &4.07  &7        \\
         &$\frac{1}{64}$    &-         &-       &1.8284e-7   &10.18 &8        \\    
\hline
  \multirow{4}{*}{0.8}                            
         &$\frac{1}{8}$     &432.29    &10333   &7.3852e-6   &0.92  &8     \\
         &$\frac{1}{16}$    &789.40    &10193   &3.3541e-6   &2.53  &10       \\
         &$\frac{1}{32}$    &1822.47   &10113   &1.5222e-6   &6.69  &12        \\
         &$\frac{1}{64}$    &-         &-       &7.0696e-7   &17.97 &14        \\
\bottomrule
    \end{tabular}
\end{table}

In this example, we use the central difference discretization for $\Delta$.
Errors, iteration numbers and CPU times of GMRES method with and without preconditioner for different $\alpha$ are listed in Tables \ref{table1} and \ref{table2}. It can be seen from these tables that the proposed preconditioning strategy can significantly reduce the iteration numbers and CPU times. In addition, since in this case, the difference of the coefficient matrix $A$ and the preconditioner $P$ lies in the matrix from temporal discretization, when $\mu$ changes, Table \ref{table2} shows that the proposed preconditioner is efficient for all $\alpha \in (0,1)$, especially when $\alpha$ is close to zero, with small and almost constant iteration numbers, which coincides with Theorem \ref{theorem_main}.


\begin{example} \label{ex_2}
{\rm

Consider the problem (\ref{main_problem}) with $\mathcal{L}=\sum\limits_{i=1}^{d}c_i\frac{\partial^{\beta_{i}}u(\boldsymbol{x},t)}{\partial|x_{i}|^{\beta_{i}}}$, $d=2$, $T=1$, $\Omega=(0,1)^2$, $c_1=c_2=T=1$ and the source term 
\begin{equation*}
\begin{aligned}
f(x_1, x_2, t)= & \frac{t^{\alpha+1}}{2 \cos \left(\beta_1 \pi / 2\right)}\times\left\{\left(\frac{2\left[x_1^{2-\beta_1}+\left(1-x_1\right)^{2-\beta_1}\right]}{\Gamma\left(3-\beta_1\right)}-\frac{12\left[x_1^{3-\beta_1}+\left(1-x_1\right)^{3-\beta_1}\right]}{\Gamma\left(4-\beta_1\right)}\right. \right.\\
& \left.\left.+\frac{24 \left[x_1^{4-\beta_1}+\left(1-x_1\right)^{4-\beta_1}\right]}{\Gamma\left(5-\beta_1\right)}\right) x_2^2\left(1-x_2\right)^2\right\}\\
&+\frac{t^{\alpha+1}}{2 \cos \left(\beta_2 \pi / 2\right)}\times\left\{\left(\frac{2\left[x_2^{2-\beta_2}+\left(1-x_2\right)^{2-\beta_2}\right]}{\Gamma\left(3-\beta_2\right)}-\frac{12\left[x_2^{3-\beta_2}+\left(1-x_2\right)^{3-\beta_2}\right]}{\Gamma\left(4-\beta_2\right)}\right. \right.\\
& \left.\left.+\frac{24 \left[x_2^{4-\beta_2}+\left(1-x_2\right)^{4-\beta_2}\right]}{\Gamma\left(5-\beta_2\right)}\right) x_1^2\left(1-x_1\right)^2\right\}\\
&+\Gamma(\alpha+2) t x_1^2\left(1-x_1\right)^2 x_2^2\left(1-x_2\right)^2.
\end{aligned}
\end{equation*}
The exact solution $u(x_1,x_2,t)= t^{\alpha+1}x_1^2(1-x_1)^2x_2^2(1-x_2)^2$.

}
\end{example}

\begin{table}[t]%
\centering
\tabcolsep=5pt
    \caption{Results of GMRES method with and without preconditioner for Example \ref{ex_2} with $\mu=\frac{1}{2^8}$.}
    \label{table3}
    \begin{tabular}{ccccccc}
  \toprule
   \multirow{2}{*}{$(\alpha,\beta_1,\beta_2)$}  & \multirow{2}{*}{$h$}  &  \multicolumn{2}{c}{GMRES} &  \multicolumn{3}{c}{PGMRES}\\  \cmidrule(r){3-4} \cmidrule(r){5-7} 
     &     &CPU(s)  &Iter &Error &CPU(s)  &Iter \\
\hline
  \multirow{4}{*}{(0.2,1.2,1.2)}
        
         &$\frac{1}{32}$   &4.62    &52    &4.0150e-6  &4.16  &8       \\
         &$\frac{1}{64}$   &17.04   &80    &9.6574e-7  &6.73  &8        \\
         &$\frac{1}{128}$  &126.67  &147   &2.3463e-7  &18.06 &8        \\
\hline
  \multirow{4}{*}{(0.2,1.5,1.5)}                                                                                               
              
         &$\frac{1}{32}$   &6.29    &73    &6.0992e-6  &3.68  &7       \\
         &$\frac{1}{64}$   &38.09   &149   &1.4586e-6  &5.97  &7       \\
         &$\frac{1}{128}$  &371.25  &439   &3.5102e-7  &16.76 &7        \\     
\hline
  \multirow{4}{*}{(0.2,1.8,1.8)}                                                                                               
              
         &$\frac{1}{32}$   &9.76    &115    &9.4207e-6  &3.29  &6       \\
         &$\frac{1}{64}$   &89.86   &426    &2.2892e-6  &5.32  &6        \\
         &$\frac{1}{128}$  &1114.53 &1350   &5.5708e-7  &14.39 &6        \\
\hline
  \multirow{4}{*}{(0.2,1.2,1.8)}                                                                                               
              
         &$\frac{1}{32}$   &10.98   &127    &7.8514e-6  &3.39  &7       \\
         &$\frac{1}{64}$   &70.40   &334    &1.9018e-6  &6.02  &7        \\
         &$\frac{1}{128}$  &788.82  &942    &4.6177e-7  &16.23 &7        \\
\hline
  \multirow{4}{*}{(0.5,1.2,1.2)}                                                                                               
          
         &$\frac{1}{32}$   &5.62    &65    &3.9312e-6  &5.84  &13       \\
         &$\frac{1}{64}$   &20.76   &98    &9.6292e-7  &10.20 &13        \\
         &$\frac{1}{128}$  &125.72  &155   &2.5035e-7  &27.34 &13        \\
\hline
  \multirow{4}{*}{(0.5,1.5,1.5)}                                                                                           
         &$\frac{1}{32}$   &8.14    &94    &5.9928e-6  &5.27  &11       \\
         &$\frac{1}{64}$   &36.81   &176   &1.4451e-6  &8.58  &11        \\
         &$\frac{1}{128}$  &349.84  &419   &3.5889e-7  &25.15 &12        \\
         \hline
  \multirow{4}{*}{(0.5,1.8,1.8)}                                                                                 
           
         &$\frac{1}{32}$   &11.07    &129    &9.2922e-6  &4.63  &10       \\
         &$\frac{1}{64}$   &86.29    &415    &2.2656e-6  &8.01  &10        \\
         &$\frac{1}{128}$  &1099.47  &1313   &5.5855e-7  &21.42 &10        \\
         \hline
\multirow{4}{*}{(0.5,1.2,1.8)}                                                                                 
           
         &$\frac{1}{32}$   &11.38   &132    &7.7118e-6  &5.30  &11       \\
         &$\frac{1}{64}$   &67.35   &323    &1.8786e-6  &8.68  &11        \\
         &$\frac{1}{128}$  &738.71  &908    &4.6625e-7  &23.63 &11        \\
         \hline
  \multirow{4}{*}{(0.8,1.2,1.2)}                                                                                 
             
         &$\frac{1}{32}$   &16.13   &179    &4.1081e-6  &13.17  &29       \\
         &$\frac{1}{64}$   &49.04   &233    &1.2678e-6  &21.48  &29        \\
         &$\frac{1}{128}$  &262.46  &315    &5.8475e-7  &58.45  &29        \\
         \hline
  \multirow{4}{*}{(0.8,1.5,1.5)}                                                                                 
           
         &$\frac{1}{32}$   &16.45   &195    &6.0331e-6  &12.02  &26       \\
         &$\frac{1}{64}$   &64.71   &306    &1.6316e-6  &19.71  &26        \\
         &$\frac{1}{128}$  &421.90  &507    &5.7903e-7  &55.67  &26        \\
         \hline
  \multirow{4}{*}{(0.8,1.8,1.8)}                                                                                 
              
         &$\frac{1}{32}$   &19.38    &228    &9.2264e-6  &10.82  &23       \\
         &$\frac{1}{64}$   &102.77   &485    &2.3635e-6  &17.58  &23        \\
         &$\frac{1}{128}$  &1018.57  &1229   &6.9569e-7  &47.94  &23        \\
 \hline
  \multirow{4}{*}{(0.8,1.2,1.8)}                                                                                 
              
         &$\frac{1}{32}$   &19.20   &230    &7.6827e-6  &11.51  &25       \\
         &$\frac{1}{64}$   &85.14   &408    &2.0320e-6  &19.04  &25        \\
         &$\frac{1}{128}$  &725.03  &912    &6.6318e-7  &51.84  &25        \\
\bottomrule
    \end{tabular}
\end{table}

\begin{table}[t]%
\centering
\tabcolsep=5pt
    \caption{Results of GMRES method with and without preconditioner for Example \ref{ex_2} with $h=h_1=h_2=\frac{1}{2^8}$.}
    \label{table4}
    \begin{tabular}{ccccccc}
  \toprule
   \multirow{2}{*}{$(\alpha,\beta_1,\beta_2)$}  & \multirow{2}{*}{$\mu$}  &  \multicolumn{2}{c}{GMRES} &  \multicolumn{3}{c}{PGMRES}\\  \cmidrule(r){3-4} \cmidrule(r){5-7} 
     &     &CPU(s)  &Iter &Error &CPU(s)  &Iter \\
\hline
  \multirow{4}{*}{(0.2,1.2,1.2)}
        
         &$\frac{1}{16}$   &47.24   &295   &7.1875e-7  &3.04  &8       \\
         &$\frac{1}{32}$   &102.47  &295   &3.4814e-7  &6.67  &8        \\
         &$\frac{1}{64}$   &225.20  &295   &1.6806e-7  &13.12 &8        \\
\hline
  \multirow{4}{*}{(0.2,1.5,1.5)}                                                                                               
              
         &$\frac{1}{16}$   &185.16  &1104    &5.1729e-7   &2.77   &7      \\      
         &$\frac{1}{32}$   &386.09  &1104    &2.5255e-7   &6.00   &7       \\     
         &$\frac{1}{64}$   &831.08  &1104    &1.2296e-7   &13.41  &8       \\     
\hline
  \multirow{4}{*}{(0.2,1.8,1.8)}                                                                                               
              
         &$\frac{1}{16}$  &705.36   &4390   &3.5758e-7  &2.43  &6        \\      
         &$\frac{1}{32}$  &1522.21  &4389   &1.7564e-7  &5.35  &6         \\     
         &$\frac{1}{64}$  &3361.46  &4388   &1.3923e-7  &10.71 &6        \\
\hline
  \multirow{4}{*}{(0.2,1.2,1.8)}                                                                                               
              
         &$\frac{1}{16}$  &500.64   &3134   &4.8155e-7  &2.77  &7        \\      
         &$\frac{1}{32}$  &1093.19  &3133   &2.3548e-7  &5.99  &7         \\     
         &$\frac{1}{64}$  &2188.51  &3133   &1.1723e-7  &12.32 &7        \\
\hline
  \multirow{4}{*}{(0.5,1.2,1.2)}                                                                                               
          
         &$\frac{1}{16}$  &42.50  &271   &3.9406e-6  &3.87  &10        \\      
         &$\frac{1}{32}$  &93.13  &269   &1.6985e-6  &8.29  &10         \\     
         &$\frac{1}{64}$  &201.33 &268   &7.1218e-7  &17.99 &11        \\
\hline
  \multirow{4}{*}{(0.5,1.5,1.5)}                                                                                           
         &$\frac{1}{16}$  &162.64  &1054   &3.0106e-6  &3.60   &9        \\      
         &$\frac{1}{32}$  &361.95  &1052   &1.3368e-6  &7.58   &9         \\     
         &$\frac{1}{64}$  &785.64  &1051   &5.7808e-7  &16.66  &10        \\
         \hline
  \multirow{4}{*}{(0.5,1.8,1.8)}                                                                                 
           
         &$\frac{1}{16}$  &678.71   &4271  &2.1875e-6  &2.96  &7        \\      
         &$\frac{1}{32}$  &1441.18  &4265  &1.0003e-6  &6.86  &8         \\     
         &$\frac{1}{64}$  &3190.21  &4263  &4.4670e-7  &13.83 &8        \\
         \hline
\multirow{4}{*}{(0.5,1.2,1.8)}                                                                                 
           
         &$\frac{1}{16}$  &470.73  &3010   &2.8386e-6  &3.49   &9        \\      
         &$\frac{1}{32}$  &996.69  &3004   &1.2695e-6  &7.59   &9         \\     
         &$\frac{1}{64}$  &2239.71 &3001   &5.5294e-7  &17.17  &10        \\
         \hline
  \multirow{4}{*}{(0.8,1.2,1.2)}                                                                                 
             
         &$\frac{1}{16}$  &40.85   &253   &1.1673e-5  &4.87   &14        \\      
         &$\frac{1}{32}$  &103.07  &293   &5.1970e-6  &12.99  &17         \\     
         &$\frac{1}{64}$  &248.44  &323   &2.2946e-6  &35.19  &21        \\
         \hline
  \multirow{4}{*}{(0.8,1.5,1.5)}                                                                                 
           
         &$\frac{1}{16}$  &159.37  &951   &8.3808e-6  &4.25  &12        \\      
         &$\frac{1}{32}$  &340.04  &929   &3.7688e-6  &11.50 &15         \\     
         &$\frac{1}{64}$  &695.05  &915   &1.6741e-6  &27.81 &18        \\
         \hline
  \multirow{4}{*}{(0.8,1.8,1.8)}                                                                                 
              
         &$\frac{1}{16}$  &635.93   &4046   &5.8496e-6  &3.95  &11        \\      
         &$\frac{1}{32}$  &1325.63  &4013   &2.6629e-6  &10.18 &13         \\     
         &$\frac{1}{64}$  &2941.60  &3995   &1.1917e-6  &25.02 &16        \\
 \hline
  \multirow{4}{*}{(0.8,1.2,1.8)}                                                                                 
              
         &$\frac{1}{16}$  &439.53  &2761   &7.8313e-6  &4.57  &12        \\      
         &$\frac{1}{32}$  &939.24  &2716   &3.5261e-6  &11.69 &15         \\     
         &$\frac{1}{64}$  &2008.58 &2692   &1.5690e-6  &27.78 &18        \\
\bottomrule
    \end{tabular}
\end{table}

In Example \ref{ex_2}, the fractional centered difference formula \cite{CD2} is adopted to discretize the multi-dimensional Riesz derivative. We list errors, iteration numbers and CPU times in Tables \ref{table3} and \ref{table4} for different choices of $\alpha, \beta_1$ and $\beta_2$. Obviously, for each case, the number of iterations and CPU time of GMRES method can be greatly reduced when the proposed preconditioner is used. Furthermore, for fixed $\alpha, \beta_1$ and $\beta_2$, no matter how $N$ and $J$ change, the number of iterations changes slightly when matrix sizes change, especially when $\alpha$ tends to zero, with small and nearly constant iteration numbers.

\begin{example} \label{ex_3}
{\rm

Consider the problem (\ref{main_problem}) with $\mathcal{L}=\sum_{i=1}^d \left(k_{i,+} \frac{\partial^{\beta_i}}{\partial_{+} x_i^{\beta_i}}+k_{i,-}\frac{\partial^{\beta_i}}{\partial_{-} x_i^{\beta_i}}\right)$, $d=2$, $T=1$, $\Omega=(0,1)^2$, $k_{1,+}=0.4$, $k_{1,-}=0.7$, $k_{2,+}=1.2$, $k_{2,-}=1.5$, $T=1$. The source term is
\begin{equation*}
\begin{aligned}
f(x_1, x_2, t)= & \frac{\Gamma(\alpha+3)}{2} t^2 x_1^4\left(1-x_1\right)^4 x_2^4\left(1-x_2\right)^4\\
&- t^{2+\alpha} \left\{  \left[k_{1,+} g(x_1,\alpha) + k_{1,-} g(1-x_1,\alpha)\right] x_2^4\left(1-x_2\right)^4\right. \\
&\left. + \left[k_{2,+} g(x_2,\beta) + k_{1,-} g(1-x_2,\beta)\right]x_1^4\left(1-x_1\right)^4  \right\}
\end{aligned}
\end{equation*}
with $g(\psi,\gamma)=\sum_{k=0}^{4}(-1)^{k} C_{4}^{k} \frac{\Gamma(9-k)}{\Gamma(9-k-\gamma)} \psi^{8-k-\gamma}$ and the exact solution $u(x_1,x_2,t)=t^{\alpha+2}x_1^4(1-x_1)^4x_2^4(1-x_2)^4$.

}
\end{example}

\begin{table}[t]%
\centering
\tabcolsep=5pt
    \caption{Results of GMRES method with and without preconditioner for Example \ref{ex_3} with $\mu=\frac{1}{2^8}$.}
    \label{table5}
    \begin{tabular}{ccccccc}
  \toprule
   \multirow{2}{*}{$(\alpha,\beta_1,\beta_2)$}  & \multirow{2}{*}{$h$}  &  \multicolumn{2}{c}{GMRES} &  \multicolumn{3}{c}{PGMRES}\\  \cmidrule(r){3-4} \cmidrule(r){5-7} 
     &     &CPU(s)  &Iter &Error &CPU(s)  &Iter \\
\hline
  \multirow{4}{*}{(0.2,1.2,1.2)}
        
         &$\frac{1}{32}$   &6.57    &67    &9.4542e-8  &7.87  &16       \\
         &$\frac{1}{64}$   &23.98   &111   &2.4070e-8  &13.11 &17        \\
         &$\frac{1}{128}$  &191.54  &235   &6.0776e-9  &35.70 &17       \\
\hline
  \multirow{4}{*}{(0.2,1.5,1.5)}                                                                                               
              
         &$\frac{1}{32}$   &7.02    &80    &9.4644e-8  &5.04  &10       \\
         &$\frac{1}{64}$   &37.08   &177   &2.3893e-8  &8.31  &10       \\
         &$\frac{1}{128}$  &321.73  &393   &6.0132e-9  &22.04 &10        \\     
\hline
  \multirow{4}{*}{(0.2,1.8,1.8)}                                                                                               
              
         &$\frac{1}{32}$   &11.09  &128    &7.6330e-8  &3.85  &7       \\
         &$\frac{1}{64}$   &79.59  &376    &1.9146e-8  &6.31  &7        \\
         &$\frac{1}{128}$  &883.68 &1102   &4.7980e-9  &16.50 &7        \\
\hline
  \multirow{4}{*}{(0.2,1.2,1.8)}                                                                                               
              
         &$\frac{1}{32}$   &11.40   &131    &6.5749e-8  &5.35  &11       \\
         &$\frac{1}{64}$   &69.93   &335    &1.6566e-8  &9.83  &12        \\
         &$\frac{1}{128}$  &707.02  &884    &4.1597e-9  &27.67 &13        \\
\hline
  \multirow{4}{*}{(0.5,1.2,1.2)}                                                                                               
          
         &$\frac{1}{32}$   &6.31    &72    &9.1469e-8  &8.14  &17       \\
         &$\frac{1}{64}$   &24.62   &117   &2.3574e-8  &13.35 &17        \\
         &$\frac{1}{128}$  &169.51  &208   &6.2520e-9  &37.40 &18        \\
\hline
  \multirow{4}{*}{(0.5,1.5,1.5)}                                                                                           
         &$\frac{1}{32}$   &7.17    &81    &9.3687e-8  &5.60  &11       \\
         &$\frac{1}{64}$   &35.77   &170   &2.3760e-8  &9.06  &11        \\
         &$\frac{1}{128}$  &314.21  &379   &6.0867e-9  &23.94 &11        \\
         \hline
  \multirow{4}{*}{(0.5,1.8,1.8)}                                                                                 
           
         &$\frac{1}{32}$   &11.09    &127    &7.6007e-8  &4.38  &9       \\
         &$\frac{1}{64}$   &77.24    &370    &1.9117e-8  &7.60  &9        \\
         &$\frac{1}{128}$  &896.11   &1089   &4.8436e-9  &20.21 &9        \\
         \hline
\multirow{4}{*}{(0.5,1.2,1.8)}                                                                                 
           
         &$\frac{1}{32}$   &11.03   &126    &6.5294e-8  &5.51  &11       \\
         &$\frac{1}{64}$   &68.19   &325    &1.6529e-8  &9.70  &12        \\
         &$\frac{1}{128}$  &712.54  &852    &4.2275e-9  &27.82 &13        \\
         \hline
  \multirow{4}{*}{(0.8,1.2,1.2)}                                                                                 
             
         &$\frac{1}{32}$   &15.06   &174    &9.0732e-8  &16.33  &34       \\
         &$\frac{1}{64}$   &47.64   &225    &2.6717e-8  &26.01  &34        \\
         &$\frac{1}{128}$  &264.59  &317    &1.0438e-8  &72.41  &34        \\
         \hline
  \multirow{4}{*}{(0.8,1.5,1.5)}                                                                                 
           
         &$\frac{1}{32}$   &10.56   &122    &9.3658e-8  &11.56  &24       \\
         &$\frac{1}{64}$   &42.87   &202    &2.4968e-8  &18.93  &24        \\
         &$\frac{1}{128}$  &324.21  &389    &7.6036e-9  &50.53  &24        \\
         \hline
  \multirow{4}{*}{(0.8,1.8,1.8)}                                                                                 
              
         &$\frac{1}{32}$   &11.82   &135    &7.6215e-8  &8.81   &18       \\
         &$\frac{1}{64}$   &75.87   &359    &1.9769e-8  &14.08  &18        \\
         &$\frac{1}{128}$  &850.05  &1035   &5.6062e-9  &37.32  &18        \\
 \hline
  \multirow{4}{*}{(0.8,1.2,1.8)}                                                                                 
              
         &$\frac{1}{32}$   &12.43   &142    &6.5619e-8  &10.51  &21       \\
         &$\frac{1}{64}$   &64.54   &304    &1.7501e-8  &16.88  &21        \\
         &$\frac{1}{128}$  &669.95  &808    &5.3575e-9  &45.32  &21        \\
\bottomrule
    \end{tabular}
\end{table}

\begin{table}[t]%
\centering
\tabcolsep=5pt
    \caption{Results of GMRES method with and without preconditioner for Example \ref{ex_3} with $h=h_1=h_2=\frac{1}{2^8}$.}
    \label{table6}
    \begin{tabular}{ccccccc}
  \toprule
   \multirow{2}{*}{$(\alpha,\beta_1,\beta_2)$}  & \multirow{2}{*}{$\mu$}  &  \multicolumn{2}{c}{GMRES} &  \multicolumn{3}{c}{PGMRES}\\  \cmidrule(r){3-4} \cmidrule(r){5-7} 
     &     &CPU(s)  &Iter &Error &CPU(s)  &Iter \\
\hline
  \multirow{4}{*}{(0.2,1.2,1.2)}
        
         &$\frac{1}{16}$   &75.63   &479   &4.7580e-9  &6.69  &17       \\
         &$\frac{1}{32}$   &160.07  &479   &2.4953e-9  &12.76 &17        \\
         &$\frac{1}{64}$   &351.72  &479   &1.8118e-9  &27.13 &17        \\
\hline
  \multirow{4}{*}{(0.2,1.5,1.5)}                                                                                               
              
         &$\frac{1}{16}$   &137.34  &862    &2.6299e-7   &3.63   &10      \\      
         &$\frac{1}{32}$   &296.36  &862    &1.8495e-7   &7.64   &10       \\     
         &$\frac{1}{64}$   &643.38  &862    &1.6097e-7   &16.27  &10       \\     
\hline
  \multirow{4}{*}{(0.2,1.8,1.8)}                                                                                               
              
         &$\frac{1}{16}$  &491.99   &3068   &1.7531e-7  &2.70  &7        \\      
         &$\frac{1}{32}$  &1127.18  &3339   &1.3689e-7  &6.00  &7         \\     
         &$\frac{1}{64}$  &2413.29  &3412   &1.2507e-7  &12.60 &7        \\
\hline
  \multirow{4}{*}{(0.2,1.2,1.8)}                                                                                               
              
         &$\frac{1}{16}$  &478.76   &2885   &1.8616e-9  &4.83  &14        \\      
         &$\frac{1}{32}$  &1066.36  &3093   &1.2904e-9  &10.38 &14         \\     
         &$\frac{1}{64}$  &2186.17  &3097   &1.1155e-9  &21.62 &14        \\
\hline
  \multirow{4}{*}{(0.5,1.2,1.2)}                                                                                               
          
         &$\frac{1}{16}$  &66.51  &435   &2.8217e-8  &6.07  &17        \\      
         &$\frac{1}{32}$  &153.31 &431   &1.1109e-8  &12.77 &17         \\     
         &$\frac{1}{64}$  &307.16 &429   &4.8879e-9  &28.39 &18        \\
\hline
  \multirow{4}{*}{(0.5,1.5,1.5)}                                                                                           
         &$\frac{1}{16}$  &134.56  &858   &1.0615e-8  &3.72   &10        \\      
         &$\frac{1}{32}$  &303.53  &858   &4.7991e-9  &7.91   &10         \\     
         &$\frac{1}{64}$  &621.95  &858   &2.6784e-9  &16.43  &10        \\
         \hline
  \multirow{4}{*}{(0.5,1.8,1.8)}                                                                                 
           
         &$\frac{1}{16}$  &496.89   &3277  &5.6702e-9  &2.77  &7        \\      
         &$\frac{1}{32}$  &1214.63  &3393  &2.8194e-9  &5.93  &7         \\     
         &$\frac{1}{64}$  &2234.08  &3116  &1.7789e-9  &13.63 &8        \\
         \hline
\multirow{4}{*}{(0.5,1.2,1.8)}                                                                                 
           
         &$\frac{1}{16}$  &427.54  &2624   &7.7390e-9  &5.01    &14        \\      
         &$\frac{1}{32}$  &967.92  &2682   &3.4586e-9  &10.64   &14         \\     
         &$\frac{1}{64}$  &2234.40 &3144   &1.9007e-9  &22.11   &14        \\
         \hline
  \multirow{4}{*}{(0.8,1.2,1.2)}                                                                                 
             
         &$\frac{1}{16}$  &52.86   &327   &1.4094e-7  &7.17   &20        \\      
         &$\frac{1}{32}$  &117.92  &341   &6.2993e-8  &18.39  &23         \\     
         &$\frac{1}{64}$  &276.53  &374   &2.8335e-8  &40.03  &26        \\
         \hline
  \multirow{4}{*}{(0.8,1.5,1.5)}                                                                                 
           
         &$\frac{1}{16}$  &128.35  &822   &4.9701e-8  &4.39  &12        \\      
         &$\frac{1}{32}$  &286.18  &820   &2.2754e-8  &10.47 &14         \\     
         &$\frac{1}{64}$  &596.25  &819   &1.0793e-8  &25.97 &17        \\
         \hline
  \multirow{4}{*}{(0.8,1.8,1.8)}                                                                                 
              
         &$\frac{1}{16}$  &468.08   &3036   &2.4875e-8  &3.45  &9        \\      
         &$\frac{1}{32}$  &1078.39  &3166   &1.1642e-8  &7.78  &10         \\     
         &$\frac{1}{64}$  &1863.26  &2569   &5.7706e-9  &20.32 &13        \\
 \hline
  \multirow{4}{*}{(0.8,1.2,1.8)}                                                                                 
              
         &$\frac{1}{16}$  &472.44  &2871   &3.6579e-8  &5.00  &14        \\      
         &$\frac{1}{32}$  &869.15  &2434   &1.6704e-8  &11.11 &15         \\     
         &$\frac{1}{64}$  &1878.46 &2628   &7.8867e-9  &23.08 &15        \\
\bottomrule
    \end{tabular}
\end{table}

The WSGD formula \cite{tian2015class} is employed for discretizing the multi-dimensional Riemann-Liouville derivative in Example \ref{ex_3}. Tables \ref{table5} and \ref{table6} show the results derived by GMRES and PGMRES methods for different values of $\alpha, \beta_1$ and $\beta_2$. Clearly, for all cases, the CPU times and the numbers of iterations have been reduced significantly when the proposed preconditioning strategy is used. Moreover, for fixed $\alpha, \beta_1,$ and $\beta_2$, the number of iterations is small and changes slightly, especially when $\alpha, \beta_1$ and $\beta_2$ are away from 1. This behavior demonstrates the efficiency and robustness of the proposed preconditioner, in agreement with the theoretical results presented in Theorem \ref{theorem_main}.

\section{Conclusions}\label{sec:conclusions}
In this paper, we initially introduce an ideal preconditioner for a class of nonsymmetric multilevel Toeplitz matrices generated by functions with essentially positive real parts. To illustrate the applicability of our proposed preconditioning approach, we specifically considered solving the all-at-once nonsymmetric multilevel Toeplitz systems derived from a broad spectrum of non-local evolutionary partial differential equations. Building upon this foundation, we then propose a novel practical PinT preconditioner based on Tau matrices. Our analysis demonstrates that the GMRES solver, when applied to these preconditioned systems, achieves an optimal convergence rate—a convergence rate independent of discretization stepsizes. Numerical experiments for the solution of various evolutionary PDEs, characterized by a small and stable number of iterations, substantiate the efficiency, robustness, and wide applicability of the proposed preconditioning strategy.

\section*{Acknowledgments}
The work of Sean Y. Hon was supported in part by the Hong Kong RGC under grant 22300921 and a start-up grant from the Croucher Foundation. The work of Siu-Long Lei was supported by the research grants MYRG2022-00262-FST and MYRG-GRG2023-00181-FST-UMDF from University of Macau.

\bibliographystyle{plain}

\end{document}